\documentclass[12pt,oneside]{article}

\usepackage{amsfonts}
\usepackage{amscd}
\usepackage{amsthm}
\usepackage{amssymb}
\usepackage{amsmath}
\usepackage{epsfig}
\usepackage[all]{xy}

\title{Marked length rigidity for one dimensional spaces}

\author{David Constantine\thanks{Department of Mathematics and Computer Science, Wesleyan University, 
265 Church St.,
Middletown, CT 06459. E-mail: \texttt{dconstantine@wesleyan.edu} } 
\and 
Jean-Fran\c{c}ois Lafont\thanks{Department of Mathematics, Ohio State University, 
231 West 18th Avenue, Columbus, OH 43210. E-mail: \texttt{jlafont@math.ohio-state.edu}}
}

\theoremstyle{definition}
\newtheorem{Def}{Definition}[section]

\theoremstyle{proposition}
\newtheorem{Lem}[Def]{Lemma}
\newtheorem{Prop}[Def]{Proposition}
\newtheorem{Claim}{Claim}

\theoremstyle{plain}

\newtheorem{Thm}[Def]{Theorem}
\newtheorem*{Theorem}{Main Theorem}

\newtheorem{Cor}[Def]{Corollary}

\begin{document}

\maketitle

\begin{abstract}

In a compact geodesic metric space of topological dimension one, the minimal length of a loop in a free homotopy class is well-defined, and provides a function $l:\pi_1(X) \longrightarrow \mathbb{R}^+\cup \infty$ (the value $\infty$ being assigned to loops which are not freely homotopic to any rectifiable loops). This function is the {\it marked length spectrum}. We introduce a subset $Conv(X)$, which is the union of all non-constant minimal loops of finite length. We show that if $X$ is a compact, non-contractible, geodesic space of topological dimension one, then $X$ deformation retracts to $Conv(X)$. Moreover, $Conv(X)$ can be characterized as the minimal subset of $X$ to which $X$ deformation retracts.
Let $X_1,X_2$ be a pair of compact, non-contractible, geodesic metric spaces of topological dimension one, and set $Y_i=Conv(X_i)$.  We prove that any isomorphism $\phi:\pi_1(X_1)\longrightarrow \pi_1(X_2)$ satisfying $l_2\circ\phi=l_1$, forces the existence of an isometry $\Phi:Y_1 \longrightarrow Y_2$ which induces the map $\phi$ on the level of fundamental groups. Thus, for compact, non-contractible, geodesic spaces of topological dimension one, the marked length spectrum completely determines the subset $Conv(X)$ up to isometry.

\end{abstract}

\section{Introduction.}

This paper is motivated by a long-standing conjecture concerning negatively curved manifolds: that the length of closed geodesics on a closed negatively curved Riemannian manifold determines the space up to isometry.  More precisely, in a closed negatively curved manifold $(M^n,g)$, there are unique geodesics in free
homotopy classes of loops. Assigning to each element in $\pi_1(M^n)$ the length of the corresponding minimal geodesic yields the function $l:\pi_1(M^n)\longrightarrow \mathbb{R}^+$, which is called the 
marked length spectrum.  The marked length spectrum conjecture states that, if we have a pair of negatively curved Riemannian metrics on $M^n$ which yield the same length function $l$, then they are in fact isometric.  In full generality, the conjecture is only known to hold for closed surfaces, which was independently established by Croke \cite{croke} and Otal \cite{otal} (see also Paulin and Hersonsky \cite{paulin-hersonsky} for some extensions to singular metrics on surfaces). In the special case where one of the Riemannian metrics is locally symmetric, this result is due to Hamenst\"adt \cite{ham} (see also Dal'bo and Kim \cite{dalbo-kim} for analogous results in the higher rank case).  

%\todo{Add the bibliography item cited above: Paulin and Hersonsky, "On the rigidity of discrete isometry groups of negatively curved spaces", Comment. Math. Helv. 72 (1997), no. 3, 349--388. } 
%DONE 8/16 -- dc

In this paper, we consider compact geodesic spaces of topological dimension one.  The starting observation is that these spaces share a lot of the properties of closed negatively curved manifolds.  In particular, they are aspherical (see Curtis and Fort \cite{curtis-fort2}), and have unique minimal length representatives in each free homotopy class of loops (by Curtis and Fort \cite{curtis-fort1}, also shown by Cannon and Conner \cite{cannon-conner}).  As such, it is reasonable to ask whether the marked length spectrum conjecture holds in the setting of compact geodesic spaces of topological dimension one.  We define a subset $Conv(X)$ of any compact geodesic space $X$ of topological dimension one. When $X$ is non-contractible, we show that $X$ deformation retracts to $Conv(X)$ (and the latter is the minimal subset with this property). We establish:

\begin{Theorem}\label{MainThm}
Let $X_1,X_2$ be a pair of compact, non-contractible, geodesic spaces of topological dimension one, and set $Y_i=Conv(X_i)$. Assume the two spaces have the same marked length spectrum, that is to say, there exists an isomorphism $\phi:\pi_1(X_1)\longrightarrow \pi_1(X_2)$ such that the following diagram commutes: 
\begin{displaymath}
	\xymatrix{ \pi_1(Y_1) \cong  \pi_1(X_1) \ar[rr]^{\phi} \ar[rd]_{l_1} & &\pi_1(X_2) \ar[ld]^{l_2} \cong \pi_1(Y_2) \\
				 &  \mathbb{R} & 
			}
\end{displaymath}
Then $Y_1$ is isometric to $Y_2$, and the isometry induces (up to change of basepoints) the isomorphism $\phi$.
\end{Theorem}

Let us provide an outline of the proof, with reference to the next section for appropriate definitions.  The idea behind the argument is to look at a certain subset of $Conv(X_1)$ consisting of {\it branch points.}  For a pair of branch points, we consider a minimal geodesic joining them, and construct a pair of geodesic loops with the property that they intersect precisely in the given minimal geodesic.  Now using the isomorphism of fundamental groups, we obtain a corresponding pair of geodesic loops in $Conv(X_2)$.  Using the fact that the lengths are preserved, we show that the corresponding pair in $Conv(X_2)$ likewise intersects in a geodesic segment, and that furthermore, the length of the intersection is exactly equal to the length of the original geodesic segment.  We then proceed to show that this correspondence is well-defined (i.e. does not depend on the pair of geodesic loops one constructs), and preserves concatenations of geodesic segments.  This is used to construct an isometry between the sets of branch points.  Using completeness, we extend this to an isometry between the closures of the sets of branch points.  Finally, we consider points in $Conv(X_1)$ which are {\it not} in the closure of the set of branch points. It is easy to see that each of these points lies on a unique maximal geodesic segment, with the property that the only branch points occur at the endpoints of the segment.  The correspondance between geodesic segments can be used to see that there is a unique, well-defined, corresponding segment in $Conv(X_2)$ of precisely the same length, allowing us to extend our isometry to all of $Conv(X_1)$. 

We conclude this introduction with a few remarks. If the spaces $X_i$ are tame (i.e. are semi-locally simply connected), then our {\bf Main Theorem} can also be deduced from some work of Culler and Morgan \cite{c-m} (see also Alperin and Bass \cite{alperin_bass}). But of course, there are numerous examples of geodesic length spaces of topological dimension one which are not semi-locally simply connected (Hawaiian Earrings, Menger curves, Sierpinski curves, etc.), for which our result does not {\it a priori} follow from theirs.  Another nice class of examples are Laakso spaces with Hausdorff dimension between one and two \cite{Laakso}.  These spaces have nice analytic properties, and work regarding the spectra of the Laplacian has been carried out on them \cite{romeo-steinhurst}.  In view of the connections between such spectra and the length spectrum in other contexts, this seems like a most interesting example.

\vskip 10pt

\centerline{\bf Acknowledgments.}

\vskip 10pt

A preliminary version of this paper was posted by the second author on the arXiv in 2003. Many thanks are due to J. W. Cannon, who found a gap in the original argument, and provided many helpful suggestions (in particular, the idea for Proposition \ref{closed}). The second author would also like to take this opportunity to thank the first author for resurrecting this project, which would otherwise have probably remained indefinitely in limbo.

\vskip 5pt

The second author was partially supported by the NSF, under grants DMS-0906483, DMS-1207782, and by an Alfred P. Sloan Research Fellowship.

\section{Preliminaries.}

In this section, we will define the various terms used in this paper, as well as quote certain results we will use in our proofs.  We start by reminding the reader of a few basic notions on length spaces (and refer to Burago, Burago and Ivanov \cite{burago} for more details on the theory).

\begin{Def}
A {\it path} in a metric space $(X,d)$, is a continuous map $f:[a,b]\longrightarrow X$ from a closed interval into $X$.  A {\it loop} in $X$ is a path $f$ satisfying $f(a)=f(b)$.  Observe that we can always view a loop as a based continuous map from $(S^1,*)$ to $(X,f(a))$.
\end{Def}

\begin{Def}
Let $(X,d)$ be a metric space.  The {\it induced length structure} is a function on the set of paths, denoted by $l_d$, and defined as follows:
    $$l_d({\bf \gamma}):=\sup \sum_{i=1}^n d({\bf \gamma}(x_{i-1}),
    {\bf \gamma}(x_i))$$
where ${\bf \gamma}:[a,b]\longrightarrow X$ is a path, and the supremum ranges over all finite collections of points $a=x_0<x_1<\ldots <x_n=b$.  A path $\gamma$ is {\it rectifiable} provided $l_d(\gamma)<\infty$.
We will often suppress the subscript and write $l(\gamma)$ if the original metric $d$ is clear from the
context. 
\end{Def}

Observe that rectifiability is preserved under finite concatenation of paths, and under restriction to subpaths. As a convenience, we will always parametrize rectifiable curves by their arclength.

\begin{Def}
Let $(X,d)$ be a metric space, and $l_d$ the induced length structure.  We define the {\it intrinsic pseudo-metric} $\hat{d}$ induced by $d$ as follows.  Let $p,q\in X$ be an arbitrary pair of points, and define the $\hat{d}$ distance between them to be $\hat{d}(p,q):=\inf l_d({\bf \gamma})$, where the infimum ranges over all paths ${\bf \gamma}:[a,b]\longrightarrow X$ satisfying $\gamma(a)=p$, $\gamma (b)=q$.
\end{Def}

Note that the function $\hat d$ actually maps $X\times X$ to $[0,\infty]$, where two points $p,q\in X$ have $\hat d(p,q)=\infty$ if and only if there are no rectifiable paths joining $p$ to $q$.

\begin{Def}\label{lengthspace}
Let $(X,d)$ be a metric space, $\hat{d}$ the corresponding intrinsic pseudo-metric.  We call $(X,d)$ a {\it length space} if $d=\hat{d}$ (in particular, $\hat d$ has image in $(0, \infty)$).
\end{Def}

\begin{Def}\label{geodspace}
Let $(X,d)$ be a length space.  We say that $(X,d)$ is a {\it geodesic space} if, for every pair of points $p,q\in X$, there is a path $\gamma_{p,q}$ joining $p$ to $q$, and having length precisely $d(p,q)$.  Such a path is called a {\it distance minimizer}. Note that such a curve is clearly locally distance minimizing.
\end{Def}

The usual terminology for such a path is geodesic.  However, we will use that term for paths which minimize length in their homotopy class (see, e.g. Definition \ref{geod} below), even though such paths need not be locally length-minimizing. In the literature, geodesic spaces are also sometimes referred to as {\it complete length spaces}.

\vskip 5pt

In topology, one of the most important concepts is that of dimension.  While there are many different notions of dimension, the one which will be of interest to us is that of {\it Lebesgue covering dimension.}
We remind the reader of the definition.

\begin{Def}
Let $X$ be a topological space.  We say that $X$ has topological (Lebesgue covering) dimension $\leq n$ if, for any open covering $\{U_i\}$ of $X$, there is a refinement $\{V_i\}$ with the property that every $x\in X$ lies in at most $n+1$ of the $V_i$.  We say that $X$ is $n$-dimensional if $X$ has dimension $\leq n$, but does not have dimension $\leq n-1$.
\end{Def}

We will denote the topological dimension of a space $X$ by $\dim(X)$.  Observe that any path connected topological space with at least two points has $\dim(X)\geq 1$.  The spaces we will be interested in are those satisfying $\dim(X)=1$.  Examples of such spaces are plentiful.  In particular, we have the following criterion (see Chapter VII in Hurewicz and Wallman \cite{hurewicz-wallman}):

\begin{Thm}
Let $(X,d)$ be a metric space, with Hausdorff dimension $\dim_H(X)$.  Consider the induced topology on $X$.  One has the inequality $\dim(X)\leq \dim_H(X)$.
\end{Thm}

Hence any path-connected metric space with Hausdorff dimension less than two (and which is not a single point) has covering dimension one.  We will henceforth focus exclusively on geodesic spaces of topological dimension one.  Note that, in general, a one dimensional geodesic space might not be negatively curved, and indeed, might not even be locally contractible.\footnote{Recall that negatively curved metric spaces are defined by the property that their universal cover satisfies the $CAT(-\delta ^2)$ condition (a metric version of negative curvature, see Bridson and Haeflinger \cite{b-h}).  Naturally, the problem in our
setting is  that spaces which are very singular might not have a universal cover.} However, such spaces often exhibit properties which are quite similar to those of negatively curved spaces.  

\begin{Def}
Let $\gamma:(S^1,*)\longrightarrow (X,p)$ be a loop.  We say that the loop is {\it reducible} provided that there is an open interval $I=(x,y)\subset S^1-\{*\}$ such that $\gamma(x)=\gamma(y)$ and the loop $\gamma|_{[x,y]}$ is nullhomotopic relative to its endpoints.  We say that $\gamma$ is {\it cyclically reducible} if the interval $I$ is allowed to include the basepoint $*$.  A loop which is not reducible (resp. cyclically reducible) will be said to be {\it reduced} (resp. {\it cyclically reduced}).  We define a constant loop to be cyclically reduced.

Similarly, let $\gamma:[a,b]\longrightarrow X$ be a path. We say that the path is {\it reducible} provided that there is an open interval $I=(x,y)\subseteq (a,b)$ such that $\gamma(x)=\gamma(y)$ and the loop $\gamma|_{[x,y]}$ is nullhomotopic relative to its endpoints. We define a constant path to be reduced.
\end{Def}

\begin{Def}\label{geod}
Let $(X,d)$ be a $1$-dimensional geodesic space.  We say that a path is {\it geodesic} provided that it is rectifiable and  minimizes the length in its homotopy class (rel. endpoints).  We say that a loop $\gamma$ is a {\it geodesic loop} provided $\gamma$ is rectifiable and has minimal length in its free homotopy class.  As per our convention, all geodesics will be parametrized by arclength. 
\end{Def}

\noindent {\bf WARNING:} As mentioned earlier, the reader who is more familiar with the Riemannian setting should beware  that, for our highly singular spaces, geodesics might not be locally length-minimizing (or even locally injective). 

\vskip 10pt

We now state Cannon and Conner's theorem (Theorem 3.9 in \cite{cannon-conner}), which establishes the uniqueness
of geodesic loops within homotopy classes:

\begin{Thm}\label{loop_unique}
Let $(X,d)$ be a compact, path-connected, $1$-dimensional metric space.  Then
        \begin{itemize}
        		\item every loop is homotopic to a reduced loop, which is unique up to reparametrization,
        		\item every loop is freely homotopic to a cyclically reduced loop, which is unique up to cyclic reparametrization.
        \end{itemize}
The homotopies are taken within the image set of the loops, so that the reduced loop (or cyclically reduced loop) always lies inside the image of the original curve.
\end{Thm}

Using this result, it is easy to show that in every free homotopy class of loops that contains a rectifiable curve, there is a unique minimal length representative.  Note that any length space is automatically a path-connected metric space.

\begin{Cor}\label{semunique}
Let $(X,d)$ be a compact $1$-dimensional length  space, $\gamma$ a loop.  Then one of the following two possibilities holds:
    \begin{itemize}
	    \item there is no rectifiable loop freely homotopic to $\gamma$,
	    \item there is a unique (up to reparametrization) minimal length rectifiable loop freely homotopic to $\gamma$.
    \end{itemize}
\end{Cor}

\begin{proof}
Let us assume that there is a rectifiable loop freely homotopic to $\gamma$, denote it by $\bar \gamma$.  Assume that $\bar \gamma$ is not cyclically reduced.  Then there exists an interval $I=[x,y]\subset S^1$ with $\bar \gamma(x)=\bar \gamma(y)$, such that the loop $\bar \gamma _1$ obtained by restricting $\bar \gamma$ to $[x,y]$ is contractible relative to its endpoints.  Consider the loop $\bar \gamma _2$ defined by restricting $\bar \gamma$ to the closure of $S^1-I$.  Then by contracting $\bar \gamma_1$, we see that $\bar \gamma$ is freely homotopic to $\bar \gamma _2$.  Furthermore, we have that $l(\bar \gamma)=l(\bar \gamma_1)+l(\bar \gamma_2)> l(\bar \gamma_2)$.  Hence $\bar \gamma$ cannot be of minimal length.

This forces any minimal length loop in a free homotopy class to be cyclically reduced.  But by the result of Cannon and Conner (Theorem \ref{loop_unique}), such a loop is unique up to reparametrization, completing the proof.
\end{proof}

In view of Corollary \ref{semunique}, when looking for geodesic loops, it is sufficient to restrict throughout to cyclically reduced loops. We can now define:

\begin{Def}\label{mls defn}
Let $(X,d)$ be a geodesic space.  Assume that, in each free homotopy class of curves on $X$, there is at most one minimal length representative.  The {\it marked length spectrum} is defined to be the function $l_d:\pi_1(X)\longrightarrow \mathbb{R}^+\cup \infty$ which assigns to each element of $\pi_1(X)$ the length of the corresponding minimal length loop (and assigns $\infty$ to the free homotopy classes that contain no rectifiable representatives).
\end{Def}

For geodesic spaces of topological dimension one, Corollary \ref{semunique} implies that the marked length spectrum is defined.

\vskip 10pt

It is straightforward to prove an analogue of the Cannon and Conner result for {\it paths} in a $1$-dimensional metric space.  

\begin{Prop}\label{path_unique}
Let $(X,d)$ be a compact, path-connected, $1$-dimensional metric space.  Then every path is homotopic (relative to the endpoints of the path) to a reduced path, which is unique up to reparametrization. Moreover, the reduced path has image lying within the image of the original path.
\end{Prop}

\begin{proof}
Let $X$ be our space, ${\bf p}$ our path, and $p,q$ the two endpoints. Consider the space $X^\prime$, defined to be the quotient space of $X$ obtained by identifying $p$ and $q$.  We claim that $X^\prime$ is a $1$-dimensional space.  Indeed, quotienting can only decrease the dimension of a space, hence $\dim(X^\prime) \leq 1$, and quotients of path-connected spaces are path-connected, so $\dim(X^\prime) \geq 1$, resulting in $X^\prime$ being $1$-dimensional.  Furthermore, $X^\prime$ is a compact, path-connected metric space, with the image ${\bf p}^\prime$ of ${\bf p}$ a closed loop based at the point $p=q$.  By the first part of Theorem \ref{loop_unique}, we have that this loop is homotopic to a unique (up to reparametrization) reduced loop. As there is a bijective correspondance between homotopies of ${\bf p}^\prime$ which preserve the basepoint and endpoint preserving homotopies of ${\bf p}$, we immediately get our claim. The statement about the images of the paths follows directly from the corresponding statement in Theorem \ref{loop_unique}.
\end{proof}

\begin{Cor}\label{I.24}
Let $(X,d)$ be a compact $1$-dimensional length space, ${\bf p}$ a rectifiable path joining points $p$ and $q$.  Then the unique reduced path  joining $p$ to $q$ homotopic (relative to the endpoints) to ${\bf p}$ has minimal length amongst all paths with this property.
\end{Cor}

\begin{proof}
This argument is identical to the one for loops: assume that ${\bf p}$ is rectifiable, but not reduced, and is a mapping from $[a,b]$ into $X$.  Then there is a subinterval $I=[a^\prime,b^\prime]\subset [a,b]$ such that ${\bf p}(a^\prime)={\bf p}(b^\prime)$ and the loop ${\bf p_1}$ obtained by restricting ${\bf p}$ to $I$ is contractible.  Denoting by ${\bf p_2}$ the concatenation of ${\bf p}$ restricted to $[a,a^\prime]$ and $[b^\prime,b]$, we see that ${\bf p_2}$ is homotopic to ${\bf p}$ via a homotopy preserving the endpoints.  Furthermore, $l({\bf p_2})< l({\bf p})$, which yields our claim.
\end{proof}

\noindent Corollary \ref{I.24} immediately yields:

\begin{Cor} \label{geodesic}
Let $(X,d)$ be a compact $1$-dimensional geodesic space, $\gamma$ a distance minimizer joining $p$ to $q$.  Then $\gamma$ is a geodesic.
\end{Cor}

\vskip 10pt

The rest of this section will be devoted to analyzing the behavior of paths under concatenation.
We fix the following notation: given a pair of paths ${\bf p}$ and ${\bf q}$, with the terminal endpoint of ${\bf p}$ coinciding with the initial endpoint of ${\bf q}$, we will denote by ${\bf q}*{\bf p}$ the concatenation of the two paths (traversing ${\bf p}$ first, followed by ${\bf q}$), and by ${\bf p}^{-1}$ the path obtained by reversing ${\bf p}$.  We will denote by $\gamma^n$ the $n$-fold iterated concatenation of a loop $\gamma$.  We start by proving several lemmas concerning concatenations of various types of paths.  Throughout the rest of this section, we will assume that $X$ is a compact, $1$-dimensional, geodesic space.

\begin{Lem}\label{concatenation}
Let ${\bf p_1}, {\bf p_2}$ be a pair of reduced paths, parametrized by arclength, in $X$. Let $t_i=l({\bf p_i})$ be the length of the respective paths, and assume that ${\bf p_1}(t_1)={\bf p_2}(0)$ (i.e. that they have a common endpoint).  Then any reduced path ${\bf q}$ homotopic to ${\bf p}:={\bf p_2}*{\bf p_1}$ is of the form ${\bf q_2}*{\bf q_1}$ where ${\bf q_i}$ is a subpath of ${\bf p_i}$.  Furthermore, we have a decomposition ${\bf p_1}={\bf r_1}*{\bf q_1}$ and ${\bf p_2}={\bf q_2}*{\bf r_2}$, satisfying ${\bf r_1}={\bf
r_2}^{-1}$.
\end{Lem}

\begin{proof}
Let us define ${\bf p}$ to be the concatenation ${\bf p_2}*{\bf p_1}$. We start by observing that the claim is trivial if the concatenation ${\bf p}$ is reduced (just take ${\bf q_i}={\bf p_i}$).  So let us assume that the concatenation ${\bf p}$ is {\it not} a reduced path, and view it as a map from $D:=[0,t_1+t_2]$ into $X$.  Since this path is reducible, there exists closed intervals $U_j\subset D$ with the property that ${\bf p}$ restricted to each $U_j$ is a closed path which is null homotopic relative to its endpoints.  Since each of the paths ${\bf p_1},{\bf p_2}$ is reduced, this forces $t_1\in U_j$.  

We now claim that, under inclusion, the family of such closed intervals forms a totally ordered set.  In order to see this, we show that any such set $U_j=[a_j,b_j]$ is a symmetric closed interval around $t_1$ (i.e. that $(a_j+b_j)/2=t_1$).  But this is clear: one can just consider the restriction of ${\bf p}$ to the two sets $[a_j,t_1]$ and $[t_1,b_j]$.  This yields a pair of paths, parametrized by arclength, joining the point ${\bf p}(a_j)={\bf p}(b_j)$ to the point ${\bf p}(t_1)$. Furthermore, each of these paths is reduced (since they are subpaths of the reduced paths ${\bf p_1}$ and ${\bf p_2}$ respectively).  But we know by Lemma \ref{path_unique} that there is a unique reduced path in each endpoint-preserving homotopy class of paths joining a pair of points. Hence the two paths have to coincide, and as they are parametrized by arclength, we immediately obtain our claim.

Next, we argue that this totally ordered chain has a maximal element. Indeed, consider the set $U$ defined to be the union of our sets $U_j$.  We claim that $U$ is still within our family.  To see this, we merely note that, by our previous observation on the $U_j=[a_j,b_j]$, the restriction of ${\bf p}$ to each symmetric (about $t_1$) subinterval of $U$ consists merely of traversing some reduced path on
$[a_j,t_1]$, followed by backtracking along the same path on $[t_1,b_j]$.  By continuity, the same must hold for the symmetric closed interval $\bar U$, so that the closure of $U$ also lies within our family.  Hence $U=\bar U$, and we have found our maximal element.

It is now easy to complete our proof: if $[a,b]$ is our maximal interval $U$, we can now define our ${\bf q_i}$ and ${\bf r_i}$ explicitely.  We set ${\bf q_1}:= {\bf p}|_{[0,a]}$, ${\bf r_1}:= {\bf p}|_{[a,t_1]}$, ${\bf r_2}:= {\bf p}|_{[t_1,b]}$, and ${\bf q_2}:= {\bf p}|_{[b,t_1+t_2]}$.  We note that it is clear that ${\bf p_1}={\bf r_1}*{\bf q_1}$ and ${\bf p_2}={\bf q_2}*{\bf r_2}$.  From our proof, it is also immediate that ${\bf r_1}={\bf r_2}^{-1}$.  Finally, since $U$ was picked to be maximal, the path ${\bf q_2}*{\bf q_1}$ must be reduced, completing the proof.
\end{proof}

An easy inductive argument gives the following corollary, which will underpin our consideration of more complicated paths later in the paper.

\begin{Cor}\label{multiple-concat}
Let $\mathbf{p_i}$ be reduced paths and assume that $\mathbf{p_{i+1}*p_i}$ is reduced for all $i$.  Then the path $\mathbf{p_n*p_{n-1}* \cdots *p_2*p_1}$ is reduced.
\end{Cor}

\begin{Cor} \label{redloop}
Let $\eta$ be a reduced loop in $X$.  Then $\eta$ can be expressed as a concatenation ${\bf p}^{-1}*\gamma*{\bf p}$, where ${\bf p}$ is a reduced path, and $\gamma$ is a geodesic loop. 
\end{Cor}

\begin{proof}
Let us view $\eta$, parametrized by arclength, as a map from $[0, t]$ into $X$.  Consider the point $p:=\eta (t/2)$, and consider the pair of paths ${\bf q_1}:= \eta|_{[0,t/2]}$ and ${\bf q_2}:=\eta_{[t/2,t]}$.  Observe that each of these paths is reduced (being a subpath of $\eta$), have common endpoints, and that $\eta = {\bf q_2}*{\bf q_1}$. Now consider the concatenation of paths ${\bf q_1}*{\bf q_2}$ and apply Lemma \ref{concatenation}.  Our claim immediately follows.
\end{proof}

For the next Lemma, we need the following definition.

\begin{Def}
A path $\mathbf{p}$ defined over times $[0,t]$ is \emph{non-self-terminating} if $\mathbf{p}(t) \notin \mathbf{p}\big({[0, t)}\big)$.  It is \emph{non-self-originating} if $\mathbf{p}(0) \notin \mathbf{p}\big({(0, t]}\big)$.
\end{Def}

We now analyze the reduced loop within the homotopy class of a path-loop-path concatenation.

\begin{Lem}\label{p-l-p}
Let $\mathbf{p}$ be a reduced path parametrized by arc length $[0,t]$, $\gamma$ a cyclically reduced loop based at $p:=\mathbf{p}(t)$, parametrized by arc length $[0,s]$. Then the unique reduced loop $\eta$ in the homotopy class of $\mathbf{p}^{-1}*\gamma*\mathbf{p}$ must pass through $p$.

Assume further that $\mathbf{p}$ is non-self-terminating.  Then $\eta$ either coincides with $\mathbf{p}$ on the interval $[0,t]$ or coincides with $\mathbf{p}^{-1}$ on the interval $[l(\eta)-t, l(\eta)]$.
\end{Lem}

\begin{proof}
By contradiction, suppose $\eta$ avoids $p$.  First, assume that $\mathbf{p}^{-1}$ has the (reduced) form $\mathbf{p}^{-1}=\mathbf{c}^{-1}*\gamma^{-n}$  for some maximal $n>0$, where $\mathbf{c}^{-1}*\mathbf{p}$ is reducible so that it avoids $p$.  Then $\mathbf{p}=\gamma^n*\mathbf{c}$ (note that the non-self-terminating assumption we will use for the second part of the Lemma rules such behavior out).  Then we reduce:
\[\mathbf{p}^{-1}*\gamma*\mathbf{p} = \mathbf{c}^{-1}*\gamma*\mathbf{c}.\]

\noindent We know that $\mathbf{c}^{-1}$ does not fully cancel $\gamma$ by the choice of $n$ maximal above.  We now may replace $\mathbf{p}$ with $\mathbf{c}$ and proceed.

So we can now assume that there is no initial factor of $\gamma^{-1}$ in $\mathbf{p}^{-1}.$  Thus $\gamma*\mathbf{p}$ must be reducible as it hits $p$.  By Lemma \ref{concatenation}, $(\gamma|_{[0, \epsilon]})^{-1}=\mathbf{p}|_{[t-\epsilon, t]}$.  Consider instead $\mathbf{p}^{-1}*\gamma$, again reducible as it hits $p$.  By Lemma \ref{concatenation}, $\gamma|_{[s-\epsilon', s]}= (\mathbf{p}^{-1}|_{[0, \epsilon']})^{-1} =\mathbf{p}|_{[t-\epsilon', t]}$.  Picking $\delta<\epsilon, \epsilon'$, we conclude that

\[(\gamma|_{[0,\delta]})^{-1} = \mathbf{p}|_{[t-\delta, t]} = \gamma|_{[s-\delta, s]}.\]

\noindent This contradicts the assumption that $\gamma$ is cyclically reduced.  Thus $p\in \eta$, completing the first part of the Lemma.

\vspace{.5cm}

For the second part of the Lemma, we now assume that $\mathbf{p}$ is non-self-terminating.  If $\mathbf{p}^{-1}*\gamma*\mathbf{p}$ is already reduced, we are done.
If not, Corollary \ref{multiple-concat} tells us that there is a reduction of $\mathbf{p}^{-1}*\gamma*\mathbf{p}$ around time $t$ or $t+s$ (as $\mathbf{p}$ and $\gamma$ are individually reduced).  By the arguments for the first part of the Lemma, if reductions at {\it both} of these times are possible, there is a contradiction to $\gamma$ being cyclically reduced.  

Let us assume first that there is no reduction around time $t$. Then $\gamma*\mathbf{p}$ is reduced, and the only way $\eta$ can fail to coincide with $\mathbf{p}$ over $[0,t]$ is if $\mathbf{p}^{-1}$ completely cancels $\gamma$ and part of $\mathbf{p}$.  But this is the situation in the first part of our arguments above; as we noted, it is ruled out by the non-self-terminating assumption.  In this case, $\eta$ coincides with $\mathbf{p}$ over $[0,t]$.

Finally, let us assume that $\mathbf{p}^{-1}*\gamma*\mathbf{p}$ does not reduce around time $t+s$, and hence $\mathbf{p}^{-1}*\gamma$ is reduced.  In this case, $\eta$ will coincide with $\mathbf{p}^{-1}$ over $[l(\eta)-t, l(\eta)]$ unless the initial segment of $\mathbf{p}^{-1}$ completely cancels $\gamma$.  Again, as noted above, this is ruled out by the non-self-terminating assumption.  Thus, the final segment of $\eta$ coincides with $\mathbf{p}^{-1}$ and the second part of the Lemma is proven.
 
\end{proof}

Note that, in our previous lemma, we can always (by reversing $\gamma$ if need be) get the reduced loop $\eta$ to coincide with ${\bf p}$ on the interval $[0,t]$.

%%%%%%%%%%%%%%%%%%%%%%%%%%%%%%%%%%%%%%%%%%%%%%%%
%%%%%%%%%%%%%%%%%%%%%%%%%%%%%%%%%%%%%%%%%%%%%%%%
%%%%%%%%%%%%%%%%%%%%%%%%%%%%%%%%%%%%%%%%%%%%%%%%

\section{$\pi_1$-hull and structure theory}

This section is devoted to understanding the structure of an arbitrary compact $1$-dimensional geodesic space $X$. In the first subsection, we will introduce the {\it $\pi_1$-hull} $Conv(X)$ of $X$, and see that $Conv(X)$ is the ``homotopically essential'' part of the space $X$. In the second subsection, we introduce the notion of {\it branch point} of $X$, and use the set of branch points in $Conv(X)$ to analyze the structure of the $\pi_1$-hull.

\subsection{Structure theory: general case}

\begin{Def}\label{hull} Given a compact geodesic space $X$ of topological dimension one, we define the {\it $\pi_1$-hull} of $X$, denoted $Conv(X)$, as the union of (the images of) all non-constant geodesic loops in $X$.  A space satisfying $X=Conv(X)$ is said to be {\it $\pi_1$-convex}.
\end{Def}

Recall that geodesic loops are both rectifiable, and cyclically reduced (and hence have minimal finite length in their free homotopy class, see Corollary \ref{semunique}). Note that, in the special case where $X$ is contractible (e.g. if $X$ is an $\mathbb R$-tree), the $\pi_1$-hull is empty. We will establish some structure theory for arbitrary compact $1$-dimensional geodesic space $X$, and show that all homotopy information about $X$ is actually carried by its $\pi_1$-hull. To begin, let us show how to extend some geodesic paths to geodesic loops.

\begin{Lem}\label{extend}
Let $\mathbf{p}$ be a non-self-terminating and non-self-originating geodesic path whose endpoints lie in $Conv(X)$.  Then $\mathbf{p}$ can be extended to a geodesic loop.  
\end{Lem}

\begin{proof}
Subdivide $\mathbf{p}$ as $\mathbf{p}=\mathbf{p}_2*\mathbf{p}_1$ meeting at $\mathbf{p}(t/2)$.  Note that $\mathbf{p}_1^{-1}$  and $\mathbf{p}_2$ are non-self-terminating.  As $\mathbf{p}(0), \mathbf{p}(t) \in Conv(X)$, there exist $\gamma_1, \gamma_2$ geodesic loops passing through (and parametrized with basepoints at) the points $\mathbf{p}(0)$ and $\mathbf{p}(t)$, respectively.  Let $\eta_1$ and $\eta_2$ be the reduced loops in the homotopy classes (based at $\mathbf{p}(t/2)$) for $\mathbf{p}_1*\gamma_1*\mathbf{p}_1^{-1}$ and $\mathbf{p}_2^{-1}*\gamma_2*\mathbf{p}_2$, respectively.  After re-orienting the $\gamma_i$ if necessary, by Lemma \ref{p-l-p} we have
\[\eta_1|_{[l(\eta_1)-t/2, l(\eta_1)]} = \mathbf{p}_1|_{[0,t/2]}\]
\[\eta_2|_{[0, t/2]} = \mathbf{p}_2|_{[0,t/2]}.\]

\noindent We can write $\eta_1=\mathbf{p}_1*\mathbf{c}_1$, $\eta_2=\mathbf{c}_2*\mathbf{p}_2$; these are both cyclically reduced.  We now have a few possible cases.

\vskip 5pt

\textbf{Case I:} $\mathbf{c}_1*\mathbf{c}_2$ is a reduced path.  Consider the closed path $\eta_2*\eta_1 = \mathbf{c}_2*\mathbf{p}_2*\mathbf{p}_1*\mathbf{c}_1$, which extends $\mathbf{p}$.  It is a reduced path as $\eta_i$ and $\mathbf{p}$ are reduced.  Under the assumption for this case, it is also cyclically reduced.  This is the geodesic loop we were seeking.

\vskip 5pt

\textbf{Case II:} $\mathbf{c}_1*\mathbf{c}_2$ is not reduced.  We will separately consider the two cyclic permutations of the path $\eta_2*\eta_1$ which might fail to be reduced paths.

\vskip 5pt

\begin{figure}
%\label{graph2}
\begin{center}
\includegraphics[width=3.5in]{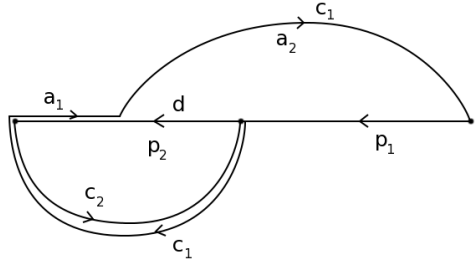}
\caption{Path configuration for Case IIa.}\label{case_2_fig}
\end{center}
\end{figure}

\textbf{Case IIa:} (See Figure \ref{case_2_fig}.)  Suppose the cyclic permutation of $\eta_2*\eta_1$ given by $\mathbf{c}_1*\mathbf{c}_2*\mathbf{p}_2*\mathbf{p}_1$ reduces to a geodesic loop not containing $\mathbf{p}$.  This happens only if $\mathbf{c}_1$ totally cancels $\mathbf{c}_2$ and then part of $\mathbf{p}$.  In such a case, we can write in reduced form $\mathbf{c}_1=\mathbf{a}*\mathbf{c}_2^{-1}$ where $\mathbf{a}*\mathbf{p}_2$ is reducible.  Write $\mathbf{a}=\mathbf{a}_2*\mathbf{a}_1$ where $\mathbf{a}_1$ is the maximal sub-path with image along $\mathbf{p}$.  We must have both $\mathbf{a}_i$ nontrivial: $\mathbf{a}_1$ since $\mathbf{a}$ partially cancels $\mathbf{p}$, and $\mathbf{a}_2$ by the fact that $\mathbf{p}_1*\mathbf{c}_1=\mathbf{p}_1*\mathbf{a}_2*\mathbf{a}_1*\mathbf{c}_2^{-1}$ is reduced  (if $\mathbf{a}_2$ is trivial, $\mathbf{p}_1*\mathbf{a}_1$ is reducible.)  Finally, let $\mathbf{d}$ be the geodesic sub-path of $\mathbf{p}$ connecting $\mathbf{p}(t/2)$ to the initial point of $\mathbf{a}_2$.

We record the following facts: $\mathbf{c}_2*\mathbf{p}_2*\mathbf{p}_1$ is reduced.  If $\mathbf{d}$ lies along $\mathbf{p}_2$, $\mathbf{d}*\mathbf{c}_2$ is reduced.  If $\mathbf{d}$ lies along $\mathbf{p}_1$, $\mathbf{d}*\mathbf{c}_2$ is still reduced, as (in this subcase) $\mathbf{c}_2$ coincides with the initial segment segment of $\mathbf{c}_1$, traversed backwards, and $\mathbf{c}_1*\mathbf{p}_1$ is cyclically reduced.  In addition, $\mathbf{a}_2*\mathbf{d}$ is reduced, by definition of $\mathbf{a}_2$.  Finally, $\mathbf{p}_1*\mathbf{a}_2$ is reduced as it is a sub-path of the reduced path $\mathbf{p}_1*\mathbf{c}_1$.

Consider, then, the closed loop $\mathbf{a}_2*\mathbf{d}*\mathbf{c}_2*\mathbf{p}_2*\mathbf{p}_1$.  It is cyclically reduced, by the remarks in the previous paragraph; hence it is a geodesic loop and it proves the Lemma.

\vskip 5pt

\textbf{Case IIb:} Suppose the cyclic permutation of $\eta_2*\eta_1$ given by $\mathbf{p}_2*\mathbf{p}_1*\mathbf{c}_1*\mathbf{c}_2$ reduces to a geodesic loop not containing $\mathbf{p}$.  Similarly to the case above, this happens only if $\mathbf{c}_2=\mathbf{c}_1^{-1}*\mathbf{a}$ in reduced form where $\mathbf{p}_1*\mathbf{a}$ is reducible.  Write $\mathbf{a}=\mathbf{a}_2*\mathbf{a}_1$ where $\mathbf{a}_2$ is the maximal sub-path of $\mathbf{a}$ along $\mathbf{p}$.  By the same arguments as in the previous case, both $\mathbf{a}_1$ and $\mathbf{a}_2$ are non-trivial.  Again, let $\mathbf{d}$ be the geodesic sub-path of $\mathbf{p}$ connecting $\mathbf{p}(t/2)$ to the endpoint of $\mathbf{a}_1$.

We record: $\mathbf{a}_1*\mathbf{p}_2*\mathbf{p}_1$ is reduced. $\mathbf{d}^{-1}*\mathbf{a}_1$ is reduced by definition of $\mathbf{a}_1$.  $\mathbf{c}_1*\mathbf{d}^{-1}$ is reduced because $\eta_1$ and $\eta_2$ are cyclically reduced.  Finally, $\mathbf{p}_1*\mathbf{c}_1$ is reduced.  

Consider $\mathbf{c}_1*\mathbf{d}^{-1}*\mathbf{a}_1*\mathbf{p}_2*\mathbf{p}_1$.  It is a geodesic loop extending $\mathbf{p}$ by the facts presented in the previous paragraph.

\end{proof}

To illustrate the usefulness of the previous lemmas, we note the following immediate corollary:

\begin{Cor} \label{cor_convex}
Suppose $X$ is a compact, 1-dimensional, geodesic metric space with $Conv(X)\neq \emptyset$.  Then $Conv(X)$ is path connected.  Furthermore, $Conv(X)$ is a strongly convex subset of $X$. \footnote{We say that a subset of a geodesic space is strongly convex provided that for every pair of points in the subset, {\it every} distance minimizer joining them also lies within the subset.}
\end{Cor}

\begin{proof}
Let $p,q\in Conv(X)$ be an arbitrary pair of distinct points, and let ${\bf p}$ be a distance minimizer joining the two points.  It is clearly non-self-terminating and non-self-originating. Since ${\bf p}$ is a distance minimizer, it is geodesic (Corollary \ref{geodesic}).  So by Lemma \ref{extend}, there is a geodesic loop extending it.  This
immediately shows that ${\bf p}$ itself lies in $Conv(X)$.  Both our claims follow.
\end{proof}

We can now establish some basic properties of the $\pi_1$-hull.

\begin{Prop}\label{closed}
Suppose $X$ is a compact, 1-dimensional, geodesic metric space with $Conv(X)\neq \emptyset$.  Then $Conv(X)$ is also a compact, 1-dimensional, geodesic metric space.\footnote{The authors are indebted to J. W. Cannon for suggesting this result and the main idea of its proof.} 
\end{Prop}

\begin{proof}
Corollary \ref{cor_convex} covers convexity; we need only show $Conv(X)$ is closed to prove the result.
Let $p\in \overline{Conv(X)}$. The proof breaks down into two cases, according to whether there is an $\epsilon$-neighborhood of $p$ containing no geodesic loop.

\vspace{.5cm}

\noindent \textbf{Case I:} For some $\epsilon>0$, the $\epsilon$-neighborhood $N$ of $p$ contains no closed geodesic. In this case, note that points in $N$ are {\it uniquely} arcwise connected (by convention, arcs will be reduced, {\it rectifiable} paths, and uniqueness is of course up to reparametrization).  For if not, we can cyclically reduce the concatenation of two such arcs to obtain a geodesic loop in $N$, a contradiction.

We must then have that for every $i$ there exists a closed geodesic $\gamma_i$ in $X$ which intersects the $\epsilon/i$-neighborhood of $p$.  Let $\alpha_i$ be a component of $N\cap \gamma_i$; let $x_i\in \alpha_i$ be a point at minimum distance from $p$.  This point is unique, as otherwise we could form a simple closed curve in $N$ by connecting two such points to $p$ with minimizing paths and to each other along $\alpha_i$.  By choosing the component $\alpha_i$ appropriately, we may assume $d(p, x_i)<\epsilon/i$.  

The points $x_i$ divide $\alpha_i$ into two arcs, $A_i$ and $B_i$.  By passing to a subsequence, we may assume that the arcs $A_i$ converge to an arc $A$ joining the exterior of $N$ to $p$, and that the $B_i$ converge to an arc $B$ doing the same.  As $N$ is uniquely arc connected, $A\cap B$ is either $p$ alone or some geodesic segment ending at $p$.

If $A \cap B = \{p\}$, take $i$ very large, so that $A_i$ and $A$ very nearly agree over a comparatively long segment of $A$.  As $X$ is a geodesic space, we must then be able to connect $A_i$ to $A$ with short geodesics near the end points. Unless $A_i$ and $A$ in fact agree over a long segment, this contradicts unique arc-connectedness of $N$. The same argument holds for $B$ and $B_i$. We claim that for large 
$i$, $A_i$ and $B_i$ coincide with $A$ and $B$ {\it all the way to $p$}. If not, then one can look at the short arc along the corresponding $\alpha_i$ where the $A_i, B_i$ differ from $A,B$. This gives a short path joining the segments $A$, $B$ together. Concatenating this path with the portion of $A$ and $B$ going to $p$ provides a closed curve in $N$ which can be shortened to a geodesic, a contradiction. We conclude that $A_i$ and $B_i$ hit $p$, and $p$ belongs to the corresponding closed geodesic $\gamma_i$, as desired.

Similarly, if $A\cap B = I$, an interval with $p$ as an endpoint, take $i$ very large so that $A_i$ and $B_i$ coincide with $I$ over a comparatively large interval.  If they coincide all the way to $x_i$, then this contradicts the fact that $\alpha_i$ was geodesic. Otherwise, we can use the portion of $\alpha_i$ near $x_i$ where they differ to obtain a geodesic loop in $N$, again a contradiction. This completes the proof of Case I.

\vspace{.5cm}

\begin{figure}
\begin{center}
\includegraphics[width=3.5in]{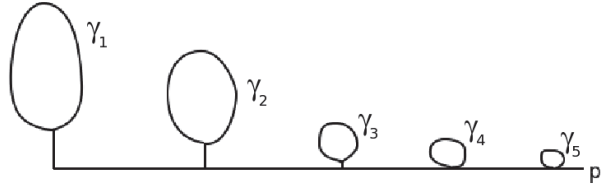}
\caption{The situation for Case II}\label{case2fig}
\end{center}
\end{figure}

\noindent \textbf{Case II:} (See Figure \ref{case2fig}.) For every $\epsilon>0$, the $\epsilon$-neighborhood $N$ of $p$ contains a closed geodesic.  If $p$ belongs to one such geodesic, we are done.  If not, connect each closed geodesic to $p$ by a distance-minimizing path.  Any two such paths must coincide on some interval with $p$ as an endpoint, otherwise $p$ belongs to a closed geodesic formed by these paths and the closed geodesics to which they connect.  Thus we may assume that $p$ lies at the endpoint of a geodesic path $\bf{p}$ to which a sequence of geodesic loops $\gamma_i$ with quickly decreasing length are connected by geodesic segments $\mathbf{t}_i$, also with quickly decreasing length.  We build a closed geodesic on which $p$ lies as follows.  Start at $p$.  Follow $\bf{p}$ to its intersection with $\mathbf{t}_1$.  Follow $\mathbf{t}_1^{-1}*\gamma_1*\mathbf{t}_1$.  Follow $\mathbf{p}$ to its intersection with $\mathbf{t}_2$ and repeat.  Continue this process; picking the sequence of paths to decrease in length sufficiently quickly gives a rectifiable curve $\bf{l}$ approaching $p$, defined over times $[0, t)$.  Set $\mathbf{l}(t)=p$. This loop is geodesic assuming each $\mathbf{t}_i$ is chosen to meet $\gamma_i$ and $\mathbf{p}$ in a single point each.

\end{proof}

Having analyzed the $\pi_1$-hull $Conv(X)$, we now turn our attention to the various connected components of $X\setminus Conv(X)$. 

\begin{Def}
For $Z$ a metric space, define an equivalence relation $\sim$ on points of $Z$ by setting $x\sim y$ if there exists a rectifiable path in $Z$ joining $x$ to $y$. A single $\sim$ equivalence class is called a {\it rectifiable component} of $Z$, and if all points in $Z$ are $\sim$ equivalent to each other, we say that $Z$ is {\it rectifiably connected}. If every point in $Z$ has a neighborhood base consisting of open, rectifiably connected sets, we say that $Z$ is {\it locally rectifiably connected}.
\end{Def}

For example, any length space is rectifiably connected. Since open metric balls in a length space are obviously rectifiably connected, length spaces are also locally rectifiably connected. 

Clearly, any rectifiable component of $Z$ is entirely contained within a single path component of $Z$. But one could a priori have a path component of $Z$ which breaks up into several distinct rectifiable components. 

\begin{Lem}\label{components}
Let $Z$ be a metric space, and consider the partition of $Z$ into 
(i) connected components, (ii) path components, and (iii) rectifiable components. If $Z$ is locally
rectifiably connected, then these three partitions of $Z$ coincide.
\end{Lem}

\begin{proof}
Since $Z$ is locally rectifiably connected, it is also locally path connected. A basic result in point set topology asserts that for locally path connected spaces, path components coincide with connected components, giving the equivalence of partitions (i) and (ii). 

For the equivalence of (ii) and (iii), observe that since $Z$ is locally rectifiably connected, each rectifiable component is open. If we had a path component $P$ of $Z$ breaking up into several rectifiable components, this would give a partition of $P$ into pairwise disjoint open sets. But since the partitions (i) and (ii) coincide, $P$ is also a connected component of $Z$. We conclude that $P$ must consist of a single rectifiable component, as desired. 
\end{proof}

\begin{Cor}\label{rec-connected}
Suppose $X$ is a compact, 1-dimensional, geodesic metric space, and let $X\setminus Conv(X) = \coprod _{i\in I}Z_i$ be the decomposition of $X\setminus Conv(X)$ into connected components. Then each $Z_i$ is rectifiably connected, and $\overline {Z_i} \cap Z_j = \emptyset$ if $i\neq j$.
\end{Cor}

\begin{proof}
$X$ is a geodesic space, so it is locally rectifiably connected. Proposition \ref{closed} tells us $Conv(X)\subseteq X$ is a closed subset. As $X\setminus Conv(X)$ is open, it inherits the property of being locally rectifiably connected. Lemma \ref{components} implies that the connected components $Z_i$ are all rectifiably connected. For the second statement, assume that $i\neq j$, and $x\in Z_j$.  As $Z_j$ is open, it is itself a neighborhood of $x$ which is disjoint from $Z_i$, and hence $x\not \in \overline{Z_i}$.
\end{proof}

\begin{Prop}\label{r-tree}
Suppose $X$ is a compact, 1-dimensional, geodesic metric space. Then each connected component of $X\setminus Conv(X)$ is a strongly convex subset of $X$, isometric to a $\mathbb R$-tree \footnote{See \cite{r_trees} for a reference on $\mathbb{R}$-trees.}.
\end{Prop}

\begin{proof}
If $Z$ is any connected component of $X\setminus Conv(X)$, Corollary \ref{rec-connected} tells us $Z$ is rectifiably connected. Let $x, y \in Z$, and let $\eta \subset X$ be any distance minimizer from $x$ to $y$. To show that $Z$ is strongly convex, we need to argue that $\eta$ lies in $Z$. By way of contradiction, let us assume that $\eta$ passes through $Conv(X)$. As $Z$ is rectifiably connected, we can find a rectifiable path $\eta ^\circ_Z \subset Z$ joining $x$ to $y$. Viewing the path $\eta ^\circ_Z$ as a path in $X$, we can apply Proposition \ref{path_unique} to obtain a reduced path $\eta_Z$ homotopic (rel. endpoints) to $\eta ^\circ_Z$. The path $\eta_Z$ has image contained within the image of $\eta ^\circ_Z$, forcing $\eta_Z \subset Z$. Now concatenating the two paths $\eta $ and $\eta_Z$ yields a closed, rectifiable, loop $\gamma^\circ = \eta*\eta_Z$. Consider the geodesic loop $\gamma$ obtained by cyclically reducing the loop $\gamma^\circ$. From the definition of $Conv(X)$, we have that $\gamma \subset Conv(X)$. But $\gamma$ was obtained by cyclically reducing the concatenation $\eta*\eta_Z$, where $\eta_Z$ was a reduced path contained entirely in $Z \subset X\setminus Conv(X)$. Since $\eta_Z$ must be fully cancelled in the cyclic reduction, but is itself a reduced path, it follows that the image of $\eta_Z$ must be contained in the image of $\eta$. The path $\eta$ is a distance minimizer, hence an embedded path. As $\eta_Z$ has image lying within the same set, and joins together the two endpoints, it must be a reparametrization of $\eta$. This yields a contradiction, as $\eta_Z \subset X \setminus Conv(X)$, while $\eta \cap Conv(X) \neq \emptyset$. We conclude that the distance minimizer $\eta$ must satisfy $\eta \subset Z$, and hence that $Z$ is indeed strongly convex.

Next we note that, given any two points $p, q \in Z$, there is a {\it unique} reduced rectifiable path $\overline{pq}$ in $Z$ (up to reparametrization) joining $p$ to $q$, i.e. $Z$ is uniquely arcwise connected. For if $\eta, \eta^\prime$ were two such paths, whose images in $Z$ did not coincide, we could cyclically reduce $\eta* \eta^\prime$ to obtain a geodesic loop in $Z$, contradicting $Z \cap Conv(X) = \emptyset$. It is now easy to see that $Z$ is a $0$-hyperbolic geodesic space: given any three points $x, y, z\in Z$, consider the distance minimizers $\overline {xy}, \overline{yz}$, and $\overline {xz}$. Reducing the concatenation $\overline {xy}* \overline{yz}$ gives us a reduced rectifiable path whose image lies within the set $\overline {xy} \cup \overline{yz}$, and joins $x$ to $z$. Since $\overline {xz}$ is another reduced rectifiable path joining these two points, the uniqueness kicks in and forces $\overline {xz}\subseteq \overline {xy} \cup \overline{yz}$. Finally, it is a well-known result that $0$-hyperbolic geodesic spaces are precisely $\mathbb R$-trees, concluding our proof.
\end{proof}

An immediate application of Proposition \ref{r-tree} is:

\begin{Cor}\label{cont-conv-nonempty}
Suppose $X$ is a compact, 1-dimensional, geodesic metric space. Then the following three statements are equivalent: (i) $X$ is contractible, (ii) $Conv(X)=\emptyset$, and (iii) $X$ is an $\mathbb R$-tree.
\end{Cor}

%\todo{Somewhat rephrased the previous corollary}

Now that we understand the connected components of $X \setminus Conv(X)$, let us see how these attach together. In view of Corollary \ref{rec-connected}, distinct connected components of $X \setminus Conv(X)$ do not interact. We now study how they attach to $Conv(X)$.

\begin{Prop}\label{attach}
Suppose $X$ is a compact, non-contractible, 1-dimensional, geodesic metric space (so $Conv(X)\neq \emptyset$). Let $X\setminus Conv(X) = \coprod _{i\in I}Z_i$ be the decomposition of $X\setminus Conv(X)$ into connected components. Then we have that each $\overline{Z_i} \cap Conv(X)$ consists of a single point $x_i$, and the (metric) completion of each $Z_i$ is precisely $Z_i \cup \{x_i\}$. Moreover, the index set $I$ is countable, and $\displaystyle \lim_{i\to \infty} diam (Z_i) = 0$.
\end{Prop}

\begin{proof}
Let $Z_i^\prime$ denote the metric completion of the space $Z_i$. Observe that the metric completion of an $\mathbb R$-tree is again an $\mathbb R$-tree (this follows easily from the $0$-hyperbolicity characterization of $\mathbb R$-trees), so $Z_i^\prime$ is a bounded, complete, $\mathbb R$-tree. Since $Z_i$ is a connected dense subset of the $\mathbb R$-tree $Z_i^\prime$, we see that for each pair of distinct points $p\neq q \in Z_i^\prime \setminus Z_i$, the distance minimizer $\overline {pq} \subset Z_i^\prime$ joining them satisfies $\overline {pq} \cap (Z_i^\prime \setminus Z_i) = \{p, q\}$, and hence $\overline {pq} \cap Z_i \neq \emptyset$.

Since $Z_i \subset X$ is strongly convex, there is a natural surjective map $\rho:Z_i^\prime \rightarrow \overline {Z_i}$ which extends the identity map on $Z_i$. Clearly $\rho$ restricts to a surjection from $Z_i^\prime\setminus Z_i$ to the set $\overline{Z_i} \cap Conv(X)$. If $p,q \in Z_i^\prime \setminus Z_i$ with $p\neq q$ satisfy $\rho(p)=\rho(q)$, then the $\rho$-image of the distance minimizer $\overline {pq} \subset Z_i^\prime$ joining $p$ to $q$ gives us a geodesic loop in $X$ which passes through points in $Z_i \subset X\setminus Conv(X)$, a contradiction. We conclude that $\rho:Z_i^\prime\setminus Z_i\rightarrow \overline{Z_i} \cap Conv(X)$ is also an {\it injective} map, and hence a bijection. 

Assume $p\neq q$ are distinct points lying in the set $\rho(Z_i^\prime\setminus Z_i)$, and let $p^\prime, q^\prime \in Z_i^\prime$ be their $\rho$-preimages. Since $p^\prime \neq q^\prime$, the distance minimizer $\eta \subset Z_i^\prime$ joining them satisfies $\eta \setminus \{p^\prime, q^\prime\} \subset Z_i$. On the other hand, the points $p \neq q$ lie in $Conv(X)$, so by Lemma \ref{cor_convex}, we can find a distance minimizer $\eta^\circ$ joining them within the set $Conv(X)$. Look at the concatenation $\rho(\eta) * \eta^\circ$ of the reduced paths $\rho(\eta)$ and $\eta^\circ$. These give a closed rectifiable loop, which in view of the discussion above (and of Lemma \ref{concatenation}) is cyclically reduced. So this defines a geodesic loop, which passes through points in $Z_i$, a contradiction. Thus $\rho(Z_i^\prime\setminus Z_i)$ consists of at most one point. But $\overline{Z_i} \setminus Z_i$ must have at least one point, for otherwise $Z_i$ would be both open and closed in the connected set $X$, a contradiction (as we are assuming that $Conv(X)\neq \emptyset$). We conclude that $\overline{Z_i} \setminus Z_i=\overline{Z_i} \cap Conv(X) $ consists of a single point $x_i$, as desired.

Finally, for each natural number $n\in \mathbb N$, consider the set $I_n \subset I$ of indices such that the corresponding connected components $Z_i$ ($i\in I_n$) have diameter $\geq 1/n$. We claim this set is {\it finite}. For if not, one has an injection $i: \mathbb N \hookrightarrow I_n$. Choose a point $x_k \in Z_{i(k)}$ with the property that $d\big(x_k, Conv(X)\big) \geq 1/n$. Then the sequence $\{x_k\}$ in $X$ has no convergent subsequence, contradicting compactness. Since $I = \bigcup _{n\in \mathbb N}I_n$ is a countable union of finite sets, it is itself countable. The statement concerning the diameters of the $Z_i$ also follows.
\end{proof}

Summarizing what we have so far, we see that an arbitrary compact $1$-dimensional geodesic space $X$ 
consists of:
\begin{itemize}
\item its $\pi_1$-hull $Conv(X)$, which is itself a ($\pi_1$-convex) compact $1$-dimensional geodesic space, sitting as a strongly convex subset of $X$ (see Corollary \ref{cor_convex} and Proposition \ref{closed}), and 
\item a countable collection of compact $\mathbb R$-trees $\overline{ Z_i}$ (whose diameters are shrinking to zero), each of which is attached to the $\pi_1$-hull $Conv(X)$ along a single terminal vertex $x_i$ (see Proposition \ref{r-tree} and Proposition \ref{attach}).
\end{itemize}
This structural result has a few nice consequences.

\begin{Cor}\label{def-retract}
Let $X$ be a compact $1$-dimensional geodesic space, and assume $X$ is not contractible (so $Conv(X)\neq \emptyset$). Then $X$ deformation retracts onto its $\pi_1$-hull $Conv(X)$. In particular, the inclusion $Conv(X) \hookrightarrow X$ induces an isomorphism $\pi_1\big(Conv(X)\big) \cong \pi_1(X)$.
\end{Cor}

\begin{proof}
Each compact $\mathbb R$-tree $\overline{ Z_i}$ deformation retracts to the corresponding terminal vertex $x_i$. It is an easy exercise to check that these homotopies glue together to define a deformation retraction of $X$ to $Conv(X)$; that the $diam(Z_i)$ shrink to zero is key to the proof.
\end{proof}

As another application, we can now provide an alternate characterization of the $\pi_1$-hull of $X$: it is the unique minimal deformation retract of $X$.

\begin{Cor}\label{characterization-conv-hull}
Let $X$ be a compact $1$-dimensional geodesic space, and assume $X$ is not contractible (so $Conv(X)\neq \emptyset$). Assume we have a subset $X_\circ \subseteq X$ satisfying the following two properties: (i) $X$ deformation retracts to $X_\circ$, and (ii) if $X$ deformation retracts to some subset $Y\subseteq X$, then $X_\circ \subseteq Y$. Then $X_\circ$ coincides with the $\pi_1$-hull $Conv(X)$.
\end{Cor} 

\begin{proof}
If a subset $X_\circ$ satisfying properties (i) and (ii) exists, it must be unique. The fact that $Conv(X)$ satisfies (i) is just Corollary \ref{def-retract} above. Now assume that $X$ deformation retracts to $Y$, and let us argue that $Conv(X)\subseteq Y$. Let $\gamma \subset X$ be an arbitrary geodesic loop. Under the deformation retraction $\rho_t:X \rightarrow Y$, the geodesic loop $\gamma$ maps to a loop $\rho_1(\gamma)$ which is freely homotopic to $\gamma$ (via the homotopy $\rho_t$). By Theorem \ref{loop_unique}, we have the containments of sets $\gamma \subseteq \rho_1(\gamma) \subseteq Y$, which gives us $\gamma \subset Y$. We conclude $Conv(X)\subseteq Y$, showing $Conv(X)$ satisfies (ii).
\end{proof}

%
%%
%%%
%%%%
%%%%%%%%%%%%%%%%%%%%%%
\subsection{Structure theory: the $\pi_1$-hull}

In the previous subsection, we reduced the study of general compact $1$-dimensional geodesic spaces to the study of their $\pi_1$-hull. In this subsection, we focus on understanding the structure of the $\pi_1$-hull $Conv(X)$. Our analysis starts with the notion of branch point.

\begin{Def}\label{branching} Let $X$ be a compact $1$-dimensional geodesic space, $p$ a point in $X$.  We say that $X$ has {\it branching} at $p$ provided there exists a triple of geodesic paths $\gamma_i:[0,\epsilon]\longrightarrow X$ with the following properties:
\begin{itemize}
	\item $\gamma_i(0)=p$ for all three paths,
	\item each concatenated path $\gamma_i*\gamma_j^{-1}$ is a reduced (and hence geodesic) path.
\end{itemize}
In other words, there are at least three distinct germs of geodesic paths originating at the point $p$. If $X$ has branching at $p$, we call $p$ a {\it branch point} of $X$.
\end{Def}

Away from the set of branch points, the local topology of $X$ is fairly simple, as indicated in the following Proposition.

\begin{Prop}\label{no-branch}
Let $X$ be a compact $1$-dimensional geodesic space, and $\mathcal B(X) \subset X$ the subset of all branch points of $X$. Assume that the point $p$ does not lie in the closure $\overline{\mathcal B(X)}$ of the set of branch points (i.e. $p\in X \setminus \overline{\mathcal B(X)}$). Then for $\epsilon$ small enough, the metric $\epsilon$-neighborhood of $p$ is isometric to either:
\begin{enumerate}
\item the half-open interval $[0, \epsilon)$, with the point $p$ corresponding to $0$, or
\item an open interval $(-\epsilon, \epsilon)$, with the point $p$ corresponding to $0$.
\end{enumerate}
\end{Prop}

\begin{proof} First, we claim that for $\delta$ small enough, the open $\delta $-ball $B_p(\delta)$ centered at $p$ is isometric to an $\mathbb R$-tree. This will follow from the fact that the $\delta $-ball contains no geodesic loops (see the end of the proof of Proposition \ref{r-tree}). Indeed, since $p$ does not lie in the closure $\overline {\mathcal B(X)}$, by choosing $\delta $ small, we can ensure that $B_p(\delta)$ contains no branch points, and that the complement $X\setminus B_p(\delta)$ is non-empty. If $\gamma$ is a geodesic loop contained in $B_p(\delta)$, take a point $x$ outside the ball, and let $\eta$ be a distance minimizer from $x$ to the curve $\gamma$. Let $L$ be the length of $\gamma$, and choose the parametrization $\gamma: [-L/2, L/2]\rightarrow X$ so that $\gamma(0)$ denotes the endpoint of $\eta$ on $\gamma$. Then the point $\gamma(0)\in B_p(\delta)$ is a branch point: the three geodesics $\gamma|_{[0, L/2]}$, $(\gamma|_{[-L/2, 0]})^{-1}$, and $\eta^{-1}$ satisfy the conditions of Definition \ref{branching}. This is a contradiction, hence $B_p(\delta)$ contains no geodesic loops, and so must be isometric to an $\mathbb R$-tree.

Now consider the connected components of $B_p(\delta) \setminus \{p\}$. If there were $\geq 3$ such connected components, take points $x_1, x_2, x_3$ in three distinct connected components, and let $\eta_i$ be a distance minimizer from $p$ to $x_i$. It is immediate that the $\eta_i$ satisfy the conditions of Definition \ref{branching}, showing that $p$ is a branch point, a contradiction. So we have that there are either one or two connected components in $B_p(\delta) \setminus \{p\}$. We consider each of these cases separately.

\vskip 10pt

\noindent{\bf Case 1:} If there is only one connected component, then $p$ must be a vertex of the $\mathbb R$-tree $B_p(\delta)$. Choose $\gamma :[0, \epsilon]\rightarrow X$ a distance minimizer from $p$ to some point in $B_p(\delta)\setminus \{p\}$. Given any point $q\in B_p(\delta)\setminus \{p\}$, we can consider the distance minimizer $\eta_q:[0, \delta^{\prime}]\rightarrow B_p(\delta)$ from $p$ to $q$. Since $p$ is a vertex of the $\mathbb R$-tree $B_p(\delta)$, we have that $\eta_q$ and $\gamma$ must coincide on some neighborhood of $p$, i.e. there exists a corresponding real number $0<\delta_q \leq \min(\epsilon , \delta^{\prime})$, with the property that $\gamma \equiv \eta_q$ on the interval $[0, \delta _q]$. If $\delta_q$ were {\it strictly} smaller than $\min(\epsilon , \delta^{\prime})$, then the point $\gamma(\delta_q)$ would be a branch point: the three geodesics $(\gamma |_{[0, \delta_q]})^{-1}$, $\gamma |_{[\delta_q, \epsilon]}$, and $\eta |_{[\delta_q, \delta^{\prime}]}$ satisfy the conditions of Definition \ref{branching}. But the neighborhood $B_p(\delta)$ was chosen to contain no branch points, forcing $\delta_q = \min(\epsilon , \delta^{\prime})$. So if $q$ is any point at distance $<\epsilon $ from $p$, then $\delta_q = \min(\epsilon , \delta^{\prime}) = \delta^{\prime}$, and we have that $\eta_q\equiv \gamma|_{[0, \delta^{\prime}]}$, i.e. the point $q$ lies on $\gamma$. This immediately implies that the metric $\epsilon$-neighborhood of $p$ consists precisely of the points along the distance minimizer $\gamma$, giving the first statement in the Proposition.

\vskip 10pt

\noindent{\bf Case 2:} If there are two connected components, choose $\gamma_1, \gamma_2 :[0, \epsilon]\rightarrow X$ to be a pair of distance minimizers from $p$ to points in the two distinct components of $B_p(\delta)\setminus \{p\}$. The concatenation $\gamma:= \gamma_1* \gamma_2^{-1}$ is a distance minimizer of length $2\epsilon$, passing through the point $p$, and entirely contained in $B_p(\delta)$. If $q$ is any point at distance $<\epsilon$ from $p$, the distance minimizer $\eta_q:[0, \delta^\prime]\rightarrow X$ must coincide with either $\gamma_1 |_{[0, \delta^\prime]}$, or with $\gamma_1 |_{[0, \delta^\prime]}$ (otherwise, as in Case 1, the first point from which they start differing would give a branch point in $B_p(\delta)$, a contradiction). We conclude that the metric $\epsilon$-neighborhood of $p$ consists precisely of the points along the distance minimizer $\gamma$, giving the second statement in the Proposition.
\end{proof}

\begin{Cor}\label{circle}
Let $X$ be a compact $1$-dimensional geodesic space, and $\mathcal B(X) \subset X$ the subset of all branch points of $X$.  
If $X$ is $\pi_1$-convex, and $\mathcal B(X)=\emptyset$, then $X$ is isometric to the circle $S^1$ of some radius $r>0$.
\end{Cor}

\begin{proof}
From Proposition \ref{no-branch}, we know that each point $p \in X$ has a (metric) neighborhood $B_\epsilon(p)$ isometric to either (i) a half-open interval, or (ii) an open interval. Since $X$ is assumed to be $\pi_1$-convex, we can rule out (i), for otherwise we could deformation retract $X$ onto a proper subset of itself, contradicting Corollary \ref{characterization-conv-hull}. Since $X$ is a compact geodesic space, it is second countable and Hausdorff. Hence it is a compact connected $1$-dimensional manifold, so must be homeomorphic to $S^1$. Finally, it is easy to see that geodesic metric space structures on $S^1$ are completely determined (up to isometry) by their diameter.
\end{proof}

With this result in hand, we can now study the complement of the set of branch points in $Conv(X)$. 

\begin{Lem}\label{paths-comp-branch}
Let $X$ be a compact $1$-dimensional geodesic space, and $\mathcal B(X) \subset X$ the subset of all branch points of $X$.  Assume $\gamma: [0, L]\rightarrow X\setminus \mathcal B(X)$ is a geodesic path. Then $\gamma$ is locally a distance minimizer, and $\gamma\big( (0,L) \big) \subset X \setminus \overline{B(X)}$. 
\end{Lem}

\begin{proof}
Take any $t\in (0, L)$, and consider the numbers $\sup _{s\in [0, t]} d\big(\gamma(t), \gamma(s)\big)$ and $\sup _{s\in [t, L]} d\big(\gamma(t), \gamma(s)\big)$. As $\gamma$ is parametrized by arclength, both these numbers are $>0$, and we can 
choose $\epsilon$ so that $2\epsilon$ is smaller than both of these. Now consider the metric $\epsilon$-ball $B_{\gamma(t)}(\epsilon)$ centered at $\gamma(t)$. We first claim that, as a set, this metric ball is entirely contained in the image of $\gamma$. If not, we can find a point $p\in X$ which satisfies $d\big(p, \gamma(t)\big) < \epsilon$, and which does {\it not} lie on the image of $\gamma$. Let $\eta$ be a distance minimizer from $p$ to the image of $\gamma$ (a compact set). By the choice of $\epsilon$, $\eta$ terminates at a point on the image of $\gamma$ which is distinct from $\gamma(0), \gamma(L)$, so yields a branch point on the image of $\gamma$. This contradicts the fact that $\gamma$ lies in $X\setminus \mathcal B(X)$. 

Now that we know that the set $B_{\gamma(t)}(\epsilon)$ is contained in the image of $\gamma$, we proceed to show that  (possibly after shrinking $\epsilon$) it in fact coincides with $\gamma \big((t-\epsilon, t+\epsilon) \big)$. Indeed, take a point $x_1$ at distance $\epsilon$ from $\gamma(0)$, and let $\eta_1:[0, \epsilon] \rightarrow X$ be a distance minimizer from $\gamma(t)$ to $x_1$. Since $\gamma(t)$ is not a branching point, $\eta_1$ must initially coincide with one of the two geodesics $\gamma|_{[t, s]} $ and $(\gamma|_{[0, t]})^{-1}$. In fact, we must have either $\eta_1 \equiv \gamma|_{[t, t+\epsilon]}$ or $\eta_1 \equiv (\gamma|_{[t-\epsilon, t]})^{-1}$, for otherwise, the first point from which they start disagreeing would be a branch point on the curve $\gamma$, a contradiction. Without loss of generality, we may now assume that $\gamma|_{[t, t+\epsilon]}$ is a distance minimizer.

Next, note that the metric ball $B_{\gamma(t)}(\epsilon/2)$ {\it cannot} consist solely of the points on the curve $\gamma|_{[t, t+\epsilon/2]}$, for otherwise $(\gamma|_{[t-\epsilon/2, t]})^{-1}$ would have to coincide with $\gamma|_{[t, t+\epsilon/2]}$, contradicting the fact that $\gamma$ is reduced. Let $x_2$ be a point in $B_{\gamma(t)}(\epsilon/2)$ which does not lie on $\gamma|_{[t, t+\epsilon/2]}$, and let $\eta_2$ be a distance minimizer from $x_2$ to the (compact) set $\gamma\big({[t, t+\epsilon]}\big)$. If $\eta_2$ terminates on a point in $\gamma\big({(t, t+\epsilon)}\big)$, we would obtain a branch point on $\gamma$, a contradiction. The triangle inequality implies that $\eta_2$ cannot terminate at the point $\gamma( t+\epsilon)$.
Hence $\eta_2$ must terminate at $\gamma(t)$. This gives us a distance minimizer $\eta_2^{-1}$ which intersects the distance minimizer $\gamma|_{[t, t+\epsilon]}$ only at their common initial point $\gamma(t)$. 

Since $\gamma(t)$ is not branching, the geodesic $(\gamma|_{[0,t]})^{-1}$ must initially coincide with either $\eta_2^{-1}$ or with $\gamma|_{[t, t+\epsilon]}$. As $\gamma$ is reduced, we see that $(\gamma|_{[0,t]})^{-1}$ must coincide with $\eta_2^{-1}$. So at the cost of further shrinking $\epsilon$, we can assume that both $\gamma|_{[t-\epsilon, t]}$ and $\gamma|_{[t, t+\epsilon]}$ are distance minimizers, that only intersect at the point $\gamma(t)$. Finally, we can conclude that their union $\gamma\big( (t-\epsilon, t+\epsilon)\big)$ is exactly the metric ball $B_{\gamma(t)}(\epsilon)$. For if not, then taking a distance minimizer $\eta_3$ from a point $x_3 \in B_{\gamma(t)}(\epsilon) \setminus \gamma\big( (t-\epsilon, t+\epsilon)\big)$ to the closest point on $\gamma\big( (t-\epsilon, t+\epsilon)\big)$ would yield a branch point on $\gamma$.

So each point in $\gamma\big( (0,L) \big)$ has a metric neighborhood isometric to an open interval contained entirely in the set $\gamma\big( (0,L) \big)\subset X\setminus \mathcal B(X)$. As this is a neighborhood which is disjoint from $\mathcal B(X)$, we conclude that each of these points lies in the complement of $\overline{B(X)}$, completing the proof.
\end{proof}

\begin{Lem}\label{components-length-space}
Suppose $X$ is a $\pi_1$-convex, compact, 1-dimensional, geodesic metric space, and let $X\setminus \overline{\mathcal B(X)} = \coprod _{i\in I}W_i$ be the decomposition of $X\setminus \overline{\mathcal B(X)}$ into connected components. Then each $W_i$ is rectifiably connected, and $\overline {W_i} \cap W_j = \emptyset$ if $i\neq j$.
\end{Lem}

\begin{proof}
The proof is identical to that of Corollary \ref{rec-connected}.
\end{proof}

The fact that the $W_i$ are rectifiably connected tells us that they have a well-defined intrinsic length space structure.
 
\begin{Lem}\label{components-comp-branch}
Suppose $X$ is a $\pi_1$-convex compact 1-dimensional geodesic metric space, and assume $\mathcal B(X)\neq \emptyset$ (so $X$ is not homeomorphic to $S^1$). Then each connected component $W$ of $X\setminus \overline{\mathcal B(X)}$, equipped with the induced intrinsic geodesic metric, is isometric to an open interval of finite length.
\end{Lem}

\begin{proof}
The argument from Corollary \ref{circle} applies verbatim to give that $W$ is a connected $1$-dimensional manifold, so is either homeomorphic to $S^1$ or to an open interval. We can rule out $S^1$, for otherwise $W$ would be both closed (being a compact subset of the Hausdorff space $X$) and open (being a connected component of the open set $X\setminus \overline{\mathcal B(X)}$) proper subset of $X$ (since by hypothesis ${\mathcal B(X)}\neq \emptyset$). But this would violate the fact that $X$ is connected. 

Next, note that Lemma \ref{components-length-space} tells us that each $W$ inherits a well-defined intrinsic length space structure. Since the $W$ is homeomorphic to an open interval, this is actually a geodesic metric space: given any two points, there is a unique (up to reparametrization) embedded path joining them, which must be rectifiable (by Lemma \ref{components-length-space}) and is of minimal length amongst all rectifiable paths joining the two points.
To conclude, we merely observe that a geodesic metric space structure on an open interval is completely determined (up to isometry) by the diameter of the interval. In the case of $W$, this diameter must be finite, for if $W$ were isometric to $\mathbb R$, then there would be no rectifiable path joining the point corresponding to $0\in \mathbb R$ with a point in $\overline{\mathcal B(X)}$, contradicting the fact that $X$ is a geodesic space.

\end{proof}

\begin{Prop}\label{complement-branch-points}
Suppose $X$ is a $\pi_1$-convex compact 1-dimensional geodesic metric space, and assume $\mathcal B(X)\neq \emptyset$ (so $X$ is not homeomorphic to $S^1$). Let $X\setminus \overline{\mathcal B(X)} = \coprod _{i\in J}W_i$ be the decomposition of $X\setminus \overline{\mathcal B(X)}$ into connected components. Then we have that each $\overline{W_i} \cap \overline{\mathcal B(X)}$ consists of at most two points. Moreover, the index set $J$ is countable, and $\displaystyle \lim_{i\to \infty} diam (W_i) = 0$.
\end{Prop}

\begin{proof}
The metric completion $W_i^\prime$ of an open interval of finite length is a closed interval of the same length. There is a natural surjective map $W_i^\prime \rightarrow \overline{W_i}$ extending the identity map on $W_i$. Since $W_i^\prime\setminus W_i$ consists of two points, and maps to $\overline{W_i} \cap \overline{\mathcal B(X)}$, the first claim follows. The argument in the last paragraph of the proof of Proposition \ref{attach} applies almost verbatim to give the
statement concerning the cardinality of the indexing set and the limit of the diameters.
\end{proof}

\begin{Lem} \label{extending-W}
Suppose $X$ is a $\pi_1$-convex compact 1-dimensional geodesic metric space, and assume $\mathcal B(X)\neq \emptyset$ (so $X$ is not homeomorphic to $S^1$). Let $W$ be a component of $X\setminus \overline{\mathcal B(X)}$ which is attached along one of its endpoints to a point $p\in \overline{\mathcal B(X)} \setminus \mathcal B(X)$. Let $\gamma: [0, r]\rightarrow X$ coincide with $\overline W$, parametrized by arclength, and satisfying $\gamma(0)=p$. Assume $\hat \gamma: [-\epsilon, r]\rightarrow X$ is any geodesic path extending the geodesic $\gamma$. Then for small enough $\epsilon^\prime$, the curve $\hat \gamma|_{[-\epsilon^\prime, 0]}$ is a distance minimizer. Moreover, there exists a strictly increasing sequence $\{t_i\}_{i \in \mathbb N} \subset [-\epsilon^\prime, 0)$ with the property that $\lim t_i =0$, and each point $\hat \gamma (t_i)$ is a branch point. 
\end{Lem}

\begin{proof}
Let $d$ denote the maximal distance from $p$ to a point on $\hat \gamma \big([-\epsilon, 0]\big)$, and choose $\epsilon_1$ so that $0<\epsilon_1 <d$. Since $p\in \overline{\mathcal B(X)} \setminus \mathcal B(X)$, there exists a branching point $x_1$ satisfying $0<d(p, x_1)<\epsilon_1$. Let $\eta_1$ be a distance minimizer from $p$ to $x_1$. Since $p$ is not a branch point, $\eta_1$ starts out coinciding with either $(\hat \gamma |_{[-\epsilon, 0]})^{-1}$ or with $\hat \gamma |_{[0, r]} \equiv \gamma$. If $\eta_1$ lies entirely along one of these curves, then $x_1$ lies on $\hat \gamma$, and we can let $t_1$ satisfy $\hat \gamma (t_1)=x_1$. Otherwise, there is a first occurrence after which the two curves are distinct. We can then set $t_1\in [-\epsilon, r]$ to be the parameter at which $\eta_1$ diverges from the $\hat \gamma$ curve. Then the point $\hat \gamma(t_i)$ is branching, by an argument identical to that in Proposition \ref{no-branch}. Note that, in both cases, $t_1$ must in fact satisfy $t_1<0$, as there are no branch points on $\gamma \equiv \hat \gamma |_{[0, r]}$. So in either case, we have obtained a $t_1<0$ with the property that $\hat \gamma (t_1)$ is branching. Note that setting $\epsilon^\prime := -t_1$, we have by construction that $\hat \gamma|_{[-\epsilon ^\prime, 0]}$ coincides with the distance minimizer $(\eta_1|_{[0, \epsilon^\prime]})^{-1}$.

By induction, assume that we have already chosen $t_1, \ldots ,t_{i-1}$, and set $\epsilon_i = -\frac{1}{2}t_{i-1}>0$. Let $x_i$ be a branch point satisfying $0 < d(p, x_i) < \epsilon_i$, and let $\eta_i$ be a distance minimizer from $p$ to $x_i$. Proceeding as in the last paragraph, we let $-t_i$ be the largest positive number so that $\eta_i|_{[0, -t_i]}\equiv (\hat \gamma |_{[t_i, 0]})^{-1}$. Then $\hat \gamma (t_i)$ is the desired branching point. This defines our sequence $\{t_i\}_{i\in \mathbb N}$. Moreover, our inductive step ensures that each $t_i$ satisfies $|t_i| < \frac{1}{2}|t_{i-1}|$, hence the increasing sequence limits to zero.
\end{proof}

\begin{Cor}\label{unique-comp-attached}
Suppose $X$ is a $\pi_1$-convex compact 1-dimensional geodesic metric space, and assume $\mathcal B(X)\neq \emptyset$ (so $X$ is not homeomorphic to $S^1$). Let $X\setminus \overline{\mathcal B(X)} = \coprod _{i\in I}W_i$ be the decomposition of $X\setminus \overline{\mathcal B(X)}$ into connected components, and denote by $\rho: \coprod _{i\in I} \partial W_i^\prime \rightarrow \overline{\mathcal B(X)}$ the attaching map from the metric completion of the $W_i$ to the set $\overline{\mathcal B(X)}$. Then for any point $x\in \overline{\mathcal B(X)} \setminus \mathcal B(X)$, we have that $|\rho^{-1}(x)| \leq 1$. 
\end{Cor}

\begin{proof}
If there is a point $x\in \overline{\mathcal B(X)} \setminus \mathcal B(X)$ with $|\rho^{-1}(x)| \geq 2$, then one can concatenate the geodesic path $\gamma$ traveling along one of the incident $W_i$ with a small path along the incident $W_j$ (note that $i=j$ could a priori happen, if both endpoints of $W_i$ are attached to the same point). The resulting extension $\hat \gamma$ contains no branch points, contradicting Lemma \ref{extending-W}.
\end{proof}

This gives us a fairly good picture of how a $\pi_1$-convex compact 1-dimensional geodesic metric space is built. Specifically, such an $X$ consists of:
\begin{itemize}
	\item its set of branch points $\mathcal B(X)$, which are points where there are $\geq 3$ germs of geodesic paths emanating from the point (see Definition \ref{branching}),
	\item the set of points $\overline{\mathcal B(X)} \setminus \mathcal B(X)$, consisting of points which are not branching, but are limits of branch points (see Figure \ref{case2fig} for an illustration of such a point),
	\item a countable collection of closed intervals $W_i^\prime$, whose lengths shrink to zero, each of which is attached to $\overline{\mathcal B(X)}$ along its endponts (see Lemma \ref{components-comp-branch} and Proposition \ref{complement-branch-points}), and
	\item each point in $\overline{\mathcal B(X)} \setminus \mathcal B(X)$ has at most {\it one} $W_i^\prime$ attached to it (see Corollary \ref{unique-comp-attached}).
\end{itemize}

%%%%%%%%%%%%%%%%%%%%%%%%%%%%%%%%%%%
%%%%%%%%%%%%%%%%%%%%%%%%%%%%%%%%%%%
%%%%%%%%%%%%%%%%%%%%%%%%%%%%%%%%%%%

\section{Marked length spectrum rigidity}

This section is devoted to proving our {\bf Main Theorem}. Let us recall the general setup. We are given two compact $1$-dimensional geodesic spaces $X_1, X_2$, and an isomorphism  $\phi:\pi_1(X_1)\longrightarrow \pi_1(X_2)$ which preserves the marked length spectrum. Since the spaces $X_i$ are by hypothesis not contractible, we know that $Conv(X_i)\neq \emptyset$ (see Corollary \ref{cont-conv-nonempty}). We want to conclude that $Conv(X_1)$ is isometric to $Conv(X_2)$. From the analysis in the last section, we know that the inclusions $j_i:Conv(X_i)\hookrightarrow X_i$ induce isomorphisms $(j_i)_*$ on $\pi_1$ (see Corollary \ref{def-retract}). It follows from the structure theory of these spaces that composing the isomorphism $(j_i)_*$ with the length function $l_i$ gives the length function on the space $Conv(X_i)$. We have established that, under the hypotheses of the {\bf Main Theorem}, we actually have an induced isomorphism $\phi: \pi_1\big(Conv(X_1)\big) \rightarrow \pi_1\big(Conv(X_2) \big)$, preserving the marked length spectrum of the $\pi_1$-convex spaces $Conv(X_i)$. Thus the {\bf Main Theorem} will follow immediately from the following special case:

\begin{Thm}\label{special-case}
Let $X_1,X_2$ be a pair of compact, geodesic spaces of topological dimension one, and
assume each $X_i$ is $\pi_1$-convex (i.e. $X_i=Conv(X_i)$). Assume the two spaces have the same marked length spectrum, that is to say, there exists an isomorphism $\phi:\pi_1(X_1)\longrightarrow \pi_1(X_2)$ satisfying $l_2\circ \phi=l_1$. Then $X_1$ is isometric to $X_2$, and the isometry induces (up to change of basepoints) the isomorphism $\phi$ on $\pi_1(X_i)$.
\end{Thm}

The rest of this section will be devoted to establishing Theorem \ref{special-case}.
We start by introducing a few definitions and proving an important lemma.

\begin{Def}
Given a geodesic path ${\bf p}$ in a $1$-dimensional geodesic space $X$ joining points $p$ to $q$, we say that a pair of geodesic loops $\gamma_1, \gamma_2$ based at $p$ and parametrized by arc-length are \emph{${\bf p}$-distinguishing} provided that $\gamma_1|_{[0,l({\bf p})]}\equiv {\bf p}\equiv \gamma_2|_{[0,l({\bf p})]}$, where $l({\bf p})$ is the length of the path ${\bf p}$ (in particular, the geodesic loops start out by respecting the orientation on ${\bf p}$). Furthermore, we require that $[0,l({\bf p})]$ be a {\it maximal} subinterval (with respect to inclusion) on which the loops $\gamma_1$ and $\gamma_2$ coincide. 

%\todo{changed "largest" to maximal wrt inclusion, as the two loops may have some overlap far away from p}

If a geodesic path ${\bf p}$ has a pair of ${\bf p}$-distinguishing geodesic loops, then we say the geodesic path 
${\bf p}$ is \emph{distinguished}. The collection of distinguished paths inside a $1$-dimensional geodesic space $X$
will be denoted by $\mathcal D (X)$.
\end{Def}

%\todo{Introduced some new terminology above. Made corresponding small modifications throughout the paper to use this new terminology.}

The importance of ${\bf p}$-distinguishing loops lies in the fact that, if $\gamma_1$ and $\gamma_2$ are ${\bf p}$ distinguishing, and if we use an overline to denote the geodesic loop freely homotopic to a given loop, then we automatically have (using Lemma \ref{concatenation}):
\[l(\overline{\gamma_2*\gamma_1^{-1}})=l(\gamma_1)+l(\gamma_2)-2l({\bf p}).\]
In particular, since the concatenated loop represents the product of the elements corresponding to $\gamma_i$ in $\pi_1(X,{\bf p}(0))$, we see that the length of the geodesic path ${\bf p}$ can be recovered from the marked length spectrum. 
It is also easy to verify that the endpoints of the path ${\bf p}$ are branching (the germs of the paths ${\bf p}$ and the appropriately oriented $\gamma_1, \gamma_2$ will all be distinct), showing the:

\begin{Lem}
If ${\bf p} \in \mathcal D(X)$ joins the points $p, q \in X$, then both endpoints are branch points, i.e. $p, q \in \mathcal B(X)$.
\end{Lem}

We do not know whether or not the collection $\mathcal D(X)$ of distinguished geodesics coincides with the
set of all geodesics whose endpoints are branch points. Our next result aims at showing that geodesics which are injective near their endpoints are indeed distinguished.

%\todo{Modified following Lemma to make it a bit more general. Please check carefully.}

\begin{Lem} \label{disting}
Let $X$ be a $\pi_1$-convex compact 1-dimensional geodesic metric space, and let ${\bf p}: [0, L] \rightarrow X$ be a geodesic path joining a pair of branch points lying in $X=Conv(X)$.  Assume that, for $\epsilon_0$ small enough, the set $[0, \epsilon_0] \cup [L-\epsilon_0, L]$ lies in the set of injectivity of the map ${\bf p}$, i.e. for any $t$ in this set, $\mathbf{p}^{-1}(\mathbf{p}(t))=\{t\}$. Then there exists a pair of ${\bf p}$-distinguishing geodesic loops, i.e. ${\bf p} \in \mathcal D (X)$.
\end{Lem}

%\todo{I clarified the definition of injective on a set.  I believe this one is what we need, not that p restricted to the given set is injective.  This rules out closed loops as distinguished paths, which I deal with below.}

\begin{proof}
Since our path ${\bf p}$ joins a pair of branch points, it is easy to see that there are a pair of geodesic paths $\gamma_1,\gamma_2$ in $X$ which intersect precisely in ${\bf p}$.  Indeed, let us start by considering $p:={\bf p}(0)$, and note that since this point is branching, there exist a triple of geodesic paths $\gamma_i:[0,\epsilon]\longrightarrow X$ emanating from the point $p$ with the property that each concatenated path $\gamma_i*\gamma_j^{-1}$ is a geodesic path. Now consider the three possible concatenations ${\bf p}*\gamma_i^{-1}$.  We claim that at least two of them have to be geodesic paths.  Indeed, if not, then two of these concatenations, say ${\bf p}*\gamma_1^{-1}$ and ${\bf p}*\gamma_2^{-1}$ have to be reducible.  Using Lemma \ref{concatenation}, this forces ${\bf p}$, $\gamma_1$, and $\gamma_2$ to coincide in a small enough interval $[0,\epsilon ^\prime]$.  But this contradicts the fact that $\gamma_2*\gamma_1^{-1}$ is reduced.  So we can extend ${\bf p}$ past $p$ in two distinct ways, and still have a reduced path.  Similarly, we can extend ${\bf p}$ past the point $q:={\bf p}(L)$ in two distinct ways, and still have a reduced path.  This gives us a pair of geodesic segments which are distinct, then come together and agree precisely along ${\bf p}$, and then separate again. 

Shrinking the two geodesic segments if need be, we can assume that they are defined on $[-\epsilon, L+\epsilon]$, and that the geodesic ${\bf p}$ corresponds to the image of $[0,L]$ in both geodesics.  Now we claim that, perhaps by further shrinking the geodesic segments $\gamma_i$, we can ensure that the reduced paths we find above are non-self-terminating and non-self-originating. Without loss of generality, assume that for our triple of geodesics $\gamma_i$, it is $\gamma_1$ and $\gamma_2$ that geodesically extend $\mathbf{p}$.  As $\mathbf{p}*\gamma_j^{-1}$ ($j=1,2$) is geodesic, from Lemma \ref{concatenation} we see that there exist arbitrarily small $t$ for which $\gamma_j(t)\neq \mathbf{p}(t)$.  We require something slightly stronger, i.e. a small value of $t$ for which $\gamma_j(t) \notin \mathbf{p}\big([0,L]\big)$.  

By way of contradiction, suppose this were not the case. We choose a $0<\delta$, with the property that $\gamma_j(t) \in \mathbf{p}\big([0,T]\big)$ for all $t\in [0, \delta]$. But recall that, by hypothesis, ${\bf p}$ is {\it injective} on $[0, \epsilon_0]$, so there is a definite positive distance between $p$ and the image set $\mathbf{p}\big([\epsilon_0,T]\big)$. So at the cost of shrinking $\delta$, we can in fact assume that the set $\gamma_j\big([0, \delta] \big)$ has image in the set $\mathbf{p}\big([0,\epsilon_0]\big)$. Then the path $\gamma_j|_{[0, \delta]}$ has image contained entirely in the embedded path $\mathbf{p}$, and satisfies $\gamma_j(0)=\mathbf{p}(0)$. Since both curves $\gamma_j$ and $\mathbf{p}$ are geodesics parametrized by arclength, this forces $\gamma_j|_{[0, \delta]} \equiv\mathbf{p}|_{[0, \delta]}$, contradicting the fact that $\mathbf{p}* \gamma_j^{-1}$ is irreducible. Thus we find arbitrarily small values of $t$ satisfying $\gamma_j(t) \notin \mathbf{p}\big([0,T]\big)$; let $\delta_0$ be such a value. The curve $\gamma_j$ might not be an embedded path, so conceivably we could have points $t\in [0, \delta_0)$ with the property that $\gamma_j(t)=\gamma_j(\delta_0)$. But the set of such values forms a closed subset of $[0, \delta_0)$, which is bounded away from zero. Hence, there exists a smallest $\delta \in (0, \delta_0]$ with the property that $\gamma_j(\delta) = \gamma_j(\delta_0)$. Then by the choice of $\delta$, we have $\gamma_j(\delta)\notin \gamma_j\big({[0, \delta)}\big)$. We conclude that the concatenation $\mathbf{p}*(\gamma_j|_{[0,\delta]})^{-1}$ is non-self-originating.

We can run a symmetric argument at the other endpoint $q=\mathbf{p}(L)$ of the path $\mathbf{p}$. This yields a short geodesic $\gamma_j^\prime|_{[0, \delta^\prime]}$ originating at $q$, with the property that the concatenation $\gamma_j^\prime|_{[0, \delta^\prime]}*\mathbf{p}$ is irreducible, and non-self-terminating. Finally, consider the concatenation $\gamma_j^\prime|_{[0, \delta^\prime]}*\mathbf{p}*(\gamma_j|_{[0,\delta]})^{-1}$. In view of our discussion above, the only way this concatenation could fail to be non-self-originating (respectively, non-self-terminating) is if $\gamma_j(\delta)\in \gamma_j^\prime|_{[0, \delta^\prime]}$ (resp. $\gamma_j^\prime(\delta^\prime)\in \gamma_j|_{[0, \delta]}$). But recall that $\delta, \delta^\prime$ could be chosen arbitrarily close to zero. By the injectivity assumption, the two endpoints $p, q$ are at a positive distance $d(p, q) >0$ apart. Now choose $\delta, \delta^\prime$ small enough to satisfy $\delta + \delta^\prime < d(p,q)$. We verify the condition for non-self-originating: if $\gamma_j(\delta)\in \gamma_j^\prime|_{[0, \delta^\prime]}$, then concatenating $\gamma_j$ with a subpath of $\gamma_j^\prime$ gives us a path of length at most $\delta + \delta^\prime < d(p,q)$ joining the points $p$ and $q$, a contradiction. The condition for non-self-terminating is completely analogous. This confirms that the concatenations $\gamma_j^\prime|_{[0, \delta^\prime]}*\mathbf{p}*(\gamma_j|_{[0,\delta]})^{-1}$ (for both $j=1, 2$) are a pair of geodesic paths which are both non-self-originating and non-self-terminating.

To finish, we want to extend these geodesic paths to closed geodesic loops.  But that is precisely what Lemma \ref{extend} guarantees (the hypotheses of the Lemma are satisfied because $X$ is $\pi_1$-convex).  Hence we obtain a pair of ${\bf p}$-distinguishing geodesic loops, and we are done. 
\end{proof}

Clearly, distance minimizers are geodesic paths which are {\it globally} injective, hence satisfy the conditions of Lemma \ref{disting}. This immediately implies:

\begin{Cor}\label{dist-min-are-disting}
Let $p, q\in \mathcal B(X)$ be a pair of distinct branch points, and let ${\bf p}$ be a distance minimizer joining $p$ to $q$.
Then ${\bf p}$ is distinguished, i.e. ${\bf p} \in \mathcal D(X)$.
\end{Cor}

Note that in Lemma \ref{disting}, the local injectivity condition in particular forces the endpoints of the curve ${\bf p}$ to be distinct. Next we consider geodesic loops based at a branch point, and provide a condition for them to be distinguished.

\begin{Lem}\label{disting-loop}
Let $X$ be a $\pi_1$-convex compact 1-dimensional geodesic metric space, and let ${\bf p}: [0, L] \rightarrow X$ be a geodesic path satisfying ${\bf p}(0)={\bf p}(L)=p \in \mathcal B(X)$. Viewing ${\bf p}$ instead as a map $(S^1, *) \rightarrow (X,p)$, assume that there exists an $\epsilon_0$ so that the $\epsilon_0$-neighborhood of the basepoint $*$ lies in the set of injectivity of the map $\mathbf{p}$ (in the sense of Lemma \ref{disting}). Then the path ${\bf p}$ is distinguished, i.e. ${\bf p}\in \mathcal D(X)$.
\end{Lem}

\begin{proof}
One loop in our distinguishing pair is $\mathbf{p}^2$.  For the second loop, invoke the local injectivity condition as in the proof of Lemma \ref{disting} to find a geodesic $\gamma_1$ leaving $p$ such that $\gamma_1*\mathbf{p}$ is non-self-terminating. By Lemma \ref{extend} we may extend it to a loop $\mathbf{c}*\gamma_1*\mathbf{p}$ which is reduced, although it may not be cyclically reduced (as $\gamma_1*\mathbf{p}$ is self-originating).  If $\mathbf{c}$ does not intersect $\mathbf{p}$ in any interval properly containing $p$, it is cyclically reduced and forms a distinguishing pair for $\mathbf{p}$ with $\mathbf{p}^2$.  If it does intersect $\mathbf{p}$ in such an interval, let $\mathbf{c}'$ be the sub-path of $c$ up to its first intersection with $\mathbf{p}$.  Then $\gamma_1^{-1}*\mathbf{c}'^{-1}*\mathbf{p}*\mathbf{c}'*\gamma_1*\mathbf{p}$ forms a distinguishing pair for $\mathbf{p}$ with $\mathbf{p}^2$.
\end{proof}

%\todo{I don't think you can get $\gamma_1 \cap \mathbf{p}=p$. For instance, the geodesic $\gamma_1$ given by the branching condition could be coming in and "bumping" $\mathbf{p}$ arbitrarily close to $p$. But the argument from Lemma \ref{disting} still allows you to arrange for $\gamma_1*\mathbf{p}$ to be non-self-terminating, and perhaps that is all you need here.}

A nice consequence of the previous few results is the following:

\begin{Cor}\label{disting-from-branch-point}
Let $X$ be a $\pi_1$-convex compact 1-dimensional geodesic metric space, and let $p\in \mathcal B(X)$ be an arbitrary 
branch point. Then there exists a distinguished geodesic ${\bf p} \in \mathcal D(X)$ originating at $p$.
\end{Cor}

\begin{proof}
If $|\mathcal B(X)|\geq 2$, then given any branch point $p\in \mathcal B(X)$, we can find a branch point $q\in \mathcal B(X)$ with $q\neq p$. Let ${\bf p}$ be a distance minimizer from $p$ to $q$, and apply Corollary \ref{dist-min-are-disting} to see that this ${\bf p}$ is a distinguished geodesic.

If $\mathcal B(X)$ consists of the single point $p$, then $\overline{\mathcal B(X)}= \mathcal B(X)= \{p\}$, and the structure theory tells us that $X \setminus \{p\}$ consists of a countable collection of open intervals, of diameter shrinking to zero, each of which is attached to $p$ at both endpoints (see Lemma \ref{components-comp-branch} and Proposition \ref{complement-branch-points}). In other words, $X$ is either a bouquet of finitely many circles (with lengths attached to each loop), or a generalized Hawaiian earring space. In either case, we can take a geodesic ${\bf p}$ in $X$ which loops through a single connected component of $X\setminus \{p\}$. Lemma \ref{disting-loop} tells us ${\bf p}$ is a distinguished geodesic.
\end{proof}

Before proving our main proposition, we give one last definition.

\begin{Def}\label{geod-incident}
Let $X$ be a compact geodesic space of topological dimension one, and ${\bf p_1}$, ${\bf p_2}$ a pair of geodesics in the space, parametrized by the intervals $[a_1,b_1]$ and $[a_2,b_2]$ respectively.  We 
say that  ${\bf p_1}$ and ${\bf p_2}$ are {\it incident}, provided that ${\bf p_1}(b_1)={\bf p_2}(a_2)$.  We say that they are {\it geodesically incident} provided that, in addition to being incident, the concatenated path ${\bf p_2}*{\bf p_1}$ is geodesic.
\end{Def}

\begin{Prop} \label{isombranch} 
Let $X_1,X_2$ be a pair of $\pi_1$-convex spaces.  If they have the same marked length spectrum, then there is an isometry from the set $\mathcal B(X_1)$ of branch points of $X_1$ to the set $\mathcal B(X_2)$ of branch points of $X_2$.
\end{Prop}

\begin{proof}
We start out by defining a length preserving map from $\mathcal D(X_1)$ to $\mathcal D(X_2)$, where we recall that $\mathcal D(X)$ denotes the set of distinguished geodesics in $X$.
Let ${\bf p}\in \mathcal D(X_1)$ be given. Then by definition, there exists a pair of ${\bf p}$-distinguishing geodesic loops; call them $\gamma_1$ and $\gamma_2$. Without loss of generality, we can assume the base point $p_1$ for $\pi_1(X_1,p_1)$ is the common vertex $\gamma_i(0)$. Corresponding to the homomorphism $\Phi:\pi_1(X_1,p_1)\longrightarrow \pi_1(X_2,p_2)$, we can find a pair of closed geodesic paths $\Phi(\gamma_1)$ and $\Phi(\gamma_2)$ (i.e. reduced paths, but not necessarily cyclically reduced) based at $p_2 \in X_2$ having precisely the same lengths of minimal representatives in their free homotopy class. We will use an overline to denote the geodesic loop (i.e. cyclically reduced loop) in the free homotopy class of a loop.  Observe that by our choice of $\gamma_1$, $\gamma_2$ being ${\bf p}$-distinguishing, we have that: 
$$l_1(\overline{\gamma_2*\gamma_1^{-1}})=l_1(\gamma_1)+l_1(\gamma_2)-2l_1({\bf p}).$$ 
Furthermore, since the isomorphism preserves the marked length spectrum, we have that 
$l_2(\overline{\Phi(\gamma_i)})=l_1(\gamma_i)$ and $l_2(\overline{\Phi(\gamma_2)*\Phi(\gamma_1^{-1})})= l_1( \overline{\gamma_2*\gamma_1^{-1}})$.
Let $\overline{\Phi(\gamma_i)}= \eta_i$, so in particular, by Corollary \ref{redloop}, we have that $\Phi(\gamma_1)=\alpha^{-1}* \eta_1 *\alpha$ and $\Phi(\gamma_2)=\beta^{-1}* \eta_2* \beta$ where $\alpha, \beta$ are geodesic paths in $X_2$.

\begin{figure}
\label{graph2}
\begin{center}
\includegraphics[width=3.5in]{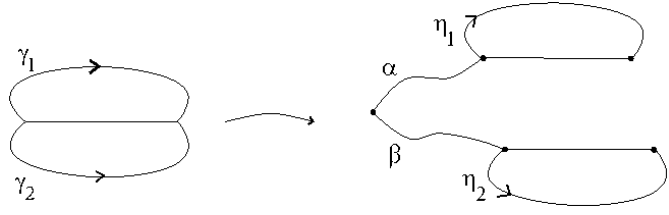}
\caption{Initial segments forced to agree: minimal representative of composite curve in first picture has shorter length than corresponding one in the second picture.}\label{fig1}
\end{center}
\end{figure}

\begin{Claim}
Using the notation from the previous paragraph, we must have: $\alpha = \mathbf{n}_2*\beta$ or $\beta=\mathbf{n}_1*\alpha$, where $\mathbf{n}_i$ is a sub-path of $\eta_i$.
\end{Claim}

Consider the path $\Phi(\gamma_2)*\Phi(\gamma_1^{-1}) = \beta^{-1} *\eta_2*\beta*\alpha^{-1}*\eta_1^{-1}*\alpha$.  Unless the concatenation $\beta*\alpha^{-1}$ completely reduces and eliminates some portion of $\eta_i$, we have the inequalities:

$$l_2(\eta_1)+l_2(\eta_2)< l_2(\overline{\Phi(\gamma_2)*\Phi(\gamma_1^{-1})})=l_1( \overline{ \gamma_2*\gamma_1^{-1}})<l_1(\gamma_1)+l_1(\gamma_2).$$  
But we have by the marked length spectrum being preserved, and the definition of $\eta_i$, that
$l_2(\eta_i)=l_2(\overline{\Phi(\gamma_i)})=l_1(\gamma_i)$ which gives us a contradiction (see Figure \ref{fig1} for an illustration of this phenomena). Denote by $q\in X_2$ the endpoint of whichever of $\alpha$ or $\beta$ contains the other.  We can, without loss of generality, assume that $p_2=q$ (by taking a change of basepoint for $\pi_1(X_2)$ if necessary).  So we have reduced to the image being a pair of geodesic loops $\eta_1$ and $\eta_2$ based at $q$.

\begin{Claim}
The geodesic loops $\eta_1$ and $\eta_2$ intersect in a path of length precisely $l_1({\bf p})$ passing through the point $q$ (i.e. $\eta_1([0,l_1({\bf p})])= \eta_2([0,l_1({\bf p})])$, but no such relationship holds for any larger interval).
\end{Claim}

In order to see this, let us assume that we can write $\eta_i= \mathbf{\sigma_i* \nu}$, where $\bf{\nu}$ is a path corresponding to the largest interval $[0,r]$ satisfying $\eta_1([0,r])=\eta_2([0,r])$, and $\sigma_i$ is the path $\eta_i([r,l_1(\gamma_i)])$ (in other words, ${\bf \nu}$ is the longest path along which the two images curves agree, and $\sigma_i$ is the rest of the respective curves).  We claim that $l_2({\bf \nu})=l_1({\bf p})$.

By our choice of $\gamma_1, \gamma_2$ being ${\bf p}$-distinguishing, we have the relation:

\begin{equation}\label{2.1} 
	2l_1({\bf p})=l_1(\gamma_1)+l_1(\gamma_2)-l_1(\overline{\gamma_2* \gamma_1^{-1}})
\end{equation}
Since we have that $\eta_i$ are the geodesic loops in the free homotopy class of $\Phi(\gamma_i)$, and as our isomorphism preserves lengths, we have that:

\begin{equation}\label{2.2} 
	l_2(\eta_i)=l_1(\gamma_i).
\end{equation}
Furthermore, the composite $\gamma_2* \gamma_1^{-1}$ corresponds to the composite $\eta_2* \eta_1^{-1}$, which forces the equality $l_1(\overline{\gamma_2* \gamma_1^{-1}})=l_2(\overline{\eta_2*
\eta_1^{-1}})$, and the latter is freely homotopic to the geodesic loop $\sigma_2*\sigma_1^{-1}$. This gives us that:

\begin{equation}\label{2.3} 
	l_1(\overline{\gamma_1*\gamma_2^{-1}})=l_2(\overline{\eta_1* \eta_2^{-1}})= l_2(\eta_1)+l_2(\eta_2)-2l_2({\bf \nu}).
\end{equation}
Substituting equations (\ref{2.2}) and (\ref{2.3}) into equation (\ref{2.1}), we obtain $2l_1({\bf p})=2l_2(\nu)$ which immediately gives us the desired equality.

\vskip 5pt

We denote the path in $X_2$ identified in this way by $\Phi^{(\gamma_1, \gamma_2)}\mathbf{p}$, in order to emphasize the dependence on the pair of ${\bf p}$-distinguishing loops $(\gamma_1, \gamma_2)$. Note that we clearly have that $\Phi^{(\gamma_1, \gamma_2)}\mathbf{p}$ lies in $\mathcal D(X_2)$, as the pair of loops $(\eta_1, \eta_2)$ are $\big(\Phi^{(\gamma_1, \gamma_2)}\mathbf{p}\big)$-distinguishing.

\vskip 5pt

\begin{Claim}
The path $\Phi^{(\gamma_1, \gamma_2)}\mathbf{p}$ is independent of the choice of ${\bf p}$-distinguishing loops. 
\end{Claim}

We have two possibilities, one of which is immediate: let $\gamma_1,\gamma_2,\gamma_3$ be geodesics based at $p_1$ an endpoint of ${\bf p}$ which are pairwise ${\bf p}$-distinguishing. Then the three image geodesic loops $\overline{\Phi (\gamma_1)}, \overline{\Phi (\gamma_2)}, \overline{\Phi (\gamma_3)}$ are all based at $q$, and pairwise  have the property that $\overline{\Phi (\gamma_i)}|_{[0,l_1({\bf p})]} =\overline{\Phi (\gamma_j)}|_{[0,l_1({\bf p})]} $. It is now immediate that all three of $\Phi^{(\gamma_i, \gamma_j)}\mathbf{p}$ must coincide.

\begin{figure}
\label{graph2.4}
\begin{center}
\includegraphics[width=3.5in]{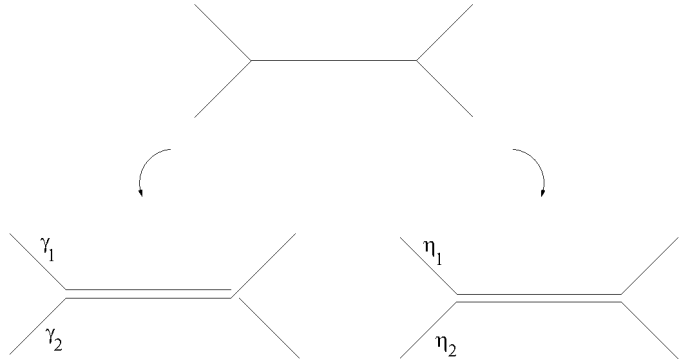}
\caption{Two pairs of ${\bf p}$-distinguishing curves with no cross-pair $\mathbf{p}$-distinguishing.}\label{fig2}
\end{center}
\end{figure}

The other possibility to account for occurs if we have two distinct pairs $(\gamma_1,\gamma_2)$ and $(\eta_1,\eta_2)$ of ${\bf p}$-distinguishing geodesics, but none of the pairs $(\gamma_i,\eta_j)$ are ${\bf p}$-distinguishing.  Since all four geodesics pass through ${\bf p}$, this means that the intersections
$\mathbf{p}_{i,j}:=\gamma_i\cap\eta_j$ are all geodesic segments which extend the original ${\bf p}$ (see Figure \ref{fig2} for an illustration of two such pairs near the geodesic).  In fact, this immediately forces the geodesic loops to have a local picture near ${\bf p}$ as in Figure \ref{fig2}.

Consider first $\overline{\Phi(\eta_1)*\Phi(\gamma_1^{-1})}$, the loop used as in Claim 2 to find $\Phi^{(\gamma_1, \eta_1)}\mathbf{p_{1,1}}$.  By Claim 1 (after fixing the basepoint for $\pi_1(X_2)$ to be the initial point of $\overline{\Phi(\gamma_1)}$) we must have $\overline{\Phi(\eta_1)}=\alpha^{-1}*\hat \eta_1*\alpha$ for a geodesic loop $\hat\eta_1$ and a geodesic $\alpha$ which lies along $\overline{\Phi(\gamma_1)}$.  By Claim 2, $\hat\eta_1$ intersects $\overline{\Phi(\gamma_1)}$ in a segment of length $l(\mathbf{p}_{1,1})$; note that if $\alpha$ is non-trivial, $\hat \eta_1$ must begin by following $\overline{\Phi(\gamma_1)}$ for a distance precisely $l(\mathbf{p}_{1,1})$.

Now undertake the same considerations for $\overline{\Phi(\eta_1)*\Phi(\gamma_2^{-1})}$.  We can maintain the same basepoint, and we know by Claim 1 that $\alpha$ lies along $\overline{\Phi(\gamma_2)}$ as well, hence is a sub-path of $\Phi^{(\gamma_1, \gamma_2)}\mathbf{p}$.  If $\alpha$ is non-trivial, $\hat \eta_1$ must begin by following $\overline{\Phi(\gamma_2)}$ for a distance precisely $l(\mathbf{p}_{2,1})$.  This however, contradicts the corresponding fact noted at the end of the previous paragraph, as $\mathbf{p}_{1,1}$ and $\mathbf{p}_{2,1}$ are both strictly longer than $\Phi^{(\gamma_1, \gamma_2)}\mathbf{p}$ and $\overline{\Phi(\gamma_1)}$ branches from $\overline{\Phi(\gamma_2)}$ after traversing $\Phi^{(\gamma_1, \gamma_2)}\mathbf{p}$. (See Figure \ref{fig5}.)

We conclude then that $\alpha$ is trivial, i.e. that $\overline{\Phi(\eta_1)}$ is a geodesic loop based at $p_2$.  Because of the branching of $\overline{\Phi(\gamma_1)}$ and $\overline{\Phi(\gamma_2)}$ at the end points of $\Phi^{(\gamma_1, \gamma_2)}\mathbf{p}$, $\Phi^{(\gamma_1, \eta_1)}\mathbf{p}_{1,1}$ and $\Phi^{(\gamma_2, \eta_1)}\mathbf{p}_{2,1}$ must be geodesic segments that extend $\Phi^{(\gamma_1, \gamma_2)}\mathbf{p}$ on opposite sides.  Applying the same arguments with $\eta_2$ replacing $\eta_1$ we see that $\overline{\Phi(\eta_1)}$ and $\overline{\Phi(\eta_2)}$ agree at least along $\Phi^{(\gamma_1, \gamma_2)}\mathbf{p}$.  But by Claim 2, this is the largest segment they can agree along, hence $\Phi^{(\gamma_1, \gamma_2)}\mathbf{p}=\Phi^{(\eta_1, \eta_2)}\mathbf{p}$, establishing Claim 3.

\begin{figure}
\begin{center}
\includegraphics[width=3.5in]{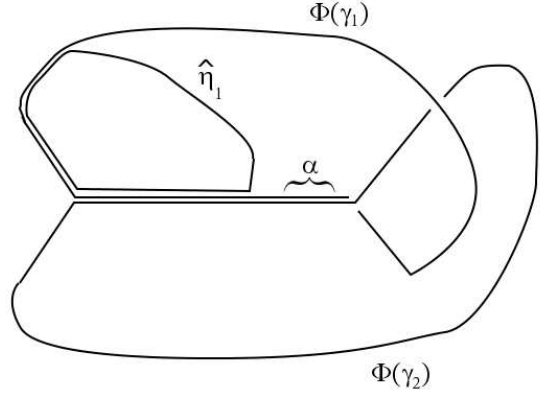}
\caption{An arrangement ruled out by the proof of Claim 3}\label{fig5}
\end{center}
\end{figure}

\vskip 10pt

We have established that the map $\Phi^{(\gamma_1, \gamma_2)}\mathbf{p}$ depends solely on the distinguished path ${\bf p}$, and not on the choice of $\mathbf{p}$-distinguishing loops $(\gamma_1, \gamma_2)$. We can now suppress the superscript, and simply write $\Phi {\bf p}$ for the image path. This gives us a well-defined map $\Phi: \mathcal D(X_1) \rightarrow D(X_2)$.

\vskip 5pt

\begin{Claim}
The map $\Phi: \mathcal D(X_1) \rightarrow D(X_2)$ is bijective.
\end{Claim}

Applying the same procedure to the inverse group homomorphism $\pi_1(X_2) \rightarrow \pi_1(X_1)$ yields a corresponding map $\Psi: \mathcal D(X_2) \rightarrow \mathcal D(X_1)$. We verify that $\Psi \circ \Phi$ is the identity map on $\mathcal D(X_1)$. Let ${\bf p}\in \mathcal D(X_1)$ be a distinguished path, and $(\gamma_1, \gamma_2)$ a pair of ${\bf p}$-distinguishing loops. From our construction, the image path $\Phi{\bf p} \in \mathcal D(X_2)$ naturally comes equipped with a pair of $\Phi{\bf p}$-distinguishing loops $(\eta_1, \eta_2)$. Recall from the argument in Claim 2 how each $\eta_i$ is obtained: we start with the corresponding $\gamma_i$, viewed as an element in $\pi_1(X_1, p)$, and use the isomorphism between fundamental groups to obtain a corresponding element in $\pi_1(X_2,q)$ (where for simplicity $p,q$ are taken to be the initial points of ${\bf p}$ and $\Phi{\bf p}$ respectively). Then $\eta_i$ is the cyclically reduced loop in the free homotopy class represented by the image element in $\pi_1(X_2, q)$. 

Reversing the argument, we see that if we start with ${\bf q}:=\Phi{\bf p} \in \mathcal D(X_2)$ and pick $(\eta_1, \eta_2)$ as the pair of ${\bf q}$-distinguishing loops, then the image $\Psi{\bf q}$ is identified via the pair $(\gamma_1, \gamma_2)$ of geodesic loops, and hence coincides with ${\bf p}$. This verifies that $\Psi \circ \Phi$ is the identity on $\mathcal D(X_1)$, and hence that $\Phi$ is injective. But an identical argument shows that $\Phi \circ \Psi$ is the identity on $\mathcal D(X_2)$, establishing that $\Phi$ is also surjective. 

\begin{Claim}
Let ${\bf p_1}, {\bf p_2} \in \mathcal D(X_1)$ be a pair of geodesically incident (see Definition \ref{geod-incident}) distinguished paths in $X_1$, meeting at a common vertex $q$ which we will take as the basepoint for $\pi_1(X_1)$. Then the corresponding pair of geodesic paths $\Phi{\bf p_1}$ and $\Phi{\bf p_2}$ is also geodesically incident.
\end{Claim}

Without loss of generality (by well-definedness of our map), we can assume that one of the closed loops used to find $\Phi{\bf p_1}, \Phi{\bf p_2} $ passes through both $\bf{p_1}$ and $\bf{p_2}$, hence can be used as one of both pairs of $\bf{p_i}$-distinguishing loops.  We refer to  Figure \ref{fig3} (top left) to illustrate our situation.  In the top left figure, we have a pair of geodesically incident paths, with the big geodesic representing the common loop, and the two smaller ones intersecting the large one in ${\bf p_1}$ and ${\bf p_2}$ respectively. The paths ${\bf p_i}$ are oriented counterclockwise along the common loop $\gamma$, so that ${\bf p_2}$ {\it precedes} ${\bf p_1}$ along $\gamma$.

Now consider the image loops (see Figure \ref{fig3}, remaining three pictures).  If the resulting curves are not incident, we have that the two geodesic loops $\overline{\Phi(\gamma_1)}$ and $\overline{\Phi(\gamma_2)}$, which must
intersect $\overline{\Phi(\gamma)}$ in geodesics segments of length $l_1(\mathbf{p_1}), l_1(\mathbf{p_2})$ respectively, have an intersection which does not represent incident subpaths of the geodesic loop $\overline{\Phi(\gamma)}$.  We have three possible cases, which we label (a), (b), (c).

\begin{figure}\label{graph2.5}
\begin{center}
\includegraphics[width=3.5in]{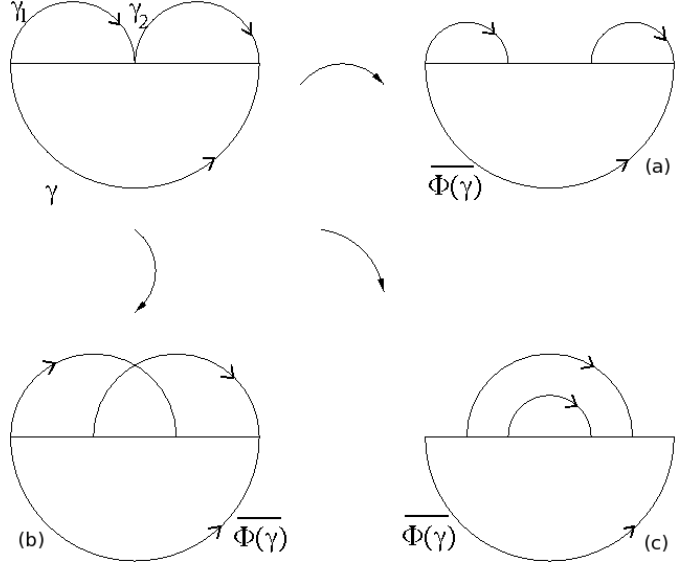}
\caption{Incidence relations are preserved: the three possible cases.}\label{fig3}
\end{center}
\end{figure}

\vskip 10pt

\noindent {\bf Case (a):} First, the intersections with $\gamma$ might be entirely disjoint (as in Figure \ref{fig3}, top right).  Following the construction of Claim 1, we may assume the basepoint $p_2$ for $\pi_1(X_2)$ is on the intersection of $\overline{\Phi(\gamma)}$ and $\overline{\Phi(\gamma_1)}$.  Then $\Phi(\gamma_2)$ has the form $\delta^{-1}*\overline{\Phi(\gamma_2)}*\delta$ for some geodesic starting at $p_2$.  Consider $\overline{\Phi(\gamma_2*\gamma^{-1})}$. Using the fact that the marked length spectrum is preserved, we have the inequalities:
$$l_2\big(\overline{\Phi(\gamma_2*\gamma^{-1})}\big) = l_1(\gamma_2*\gamma^{-1}) <
l_1(\gamma)+l_1(\gamma_2) = l_2(\overline{\Phi(\gamma)})+l_2(\overline{\Phi(\gamma_2)})$$
This forces $\delta$ to lie entirely along $\overline{\Phi(\gamma)}$.  Let $d$ denote the distance along $\delta$ between $\Phi\mathbf{p_1}$ and $\Phi\mathbf{p_2}$ and note that we have:
\[l_2(\overline{ \Phi(\gamma_1)*\Phi(\gamma_2)})=l_2(\overline{\Phi(\gamma_1)})+l_2(\overline{\Phi(\gamma_2)})+2d\]
But observe that the the geodesic loop $\overline{\gamma_1*\gamma_2}$ has length which is bounded above by $l_1(\gamma_1)+l_1(\gamma_2)$. Combined with the fact that the isomorphism preserves the marked length spectrum, this gives us a contradiction if $d>0$.  Thus $\Phi\mathbf{p_1}$ and $\Phi\mathbf{p_2}$ are adjacent on $\overline{\Phi(\gamma)}$.  

Finally, suppose their order is reversed, i.e. $\Phi\mathbf{p_1}$ precedes $\Phi\mathbf{p_2}$ when following $\overline{\Phi(\gamma)}$.  In $X_1$ we have that 
\[l_1(\gamma_2*\gamma_1*\gamma^{-1}) = l_1(\gamma)+l_1(\overline{\gamma_1*\gamma_2})-2l_1(\mathbf{p_1})-2l_1(\mathbf{p_2}).\]
Whereas in $X_2$, the corresponding loop satisfies
\[l_2(\overline{\Phi(\gamma_2)*\Phi(\gamma_1)*\Phi(\gamma^{-1})}) = l_2(\overline{\Phi(\gamma)})+l_2(\overline{\Phi(\gamma_1)})+l_2(\overline{\Phi(\gamma_2)}),\]
a strictly greater length, providing a contradiction.  We conclude that the paths $\Phi\mathbf{p_1}$ and $\Phi\mathbf{p_2}$ are geodesically incident in $X_2$ in the same order that they are in $X_1$.

\vskip 10pt

\noindent {\bf Case (b):}
The second possibility is that $\overline{\Phi(\gamma_1)}$ and $\overline{\Phi(\gamma_2)}$ intersect in a subinterval of the geodesic loop $\overline{\Phi(\gamma)}$ of length $d$ (as in Figure \ref{fig3}, bottom left).  As above, we may place the basepoint for $\pi_1(X_2)$ on $\overline{\Phi(\gamma_1)}$ and write $\Phi(\gamma_2) = \alpha^{-1}*\overline{\Phi(\gamma_2)}*\alpha$ with $\alpha$ a geodesic along $\overline{\Phi(\gamma)}$.  We consider the geodesic loop corresponding to  $\gamma_2^{-1}*\gamma*\gamma_1^{-1}$, and observe that it has length:

\begin{equation}
\label{2.4} l_1(\overline{\gamma_2^{-1}*\gamma*\gamma_1^{-1}}) = l_1(\gamma)+ l_1(\overline{\gamma_2*\gamma_1})-2l_1({\bf p_1})-2l_1({\bf p_2})
\end{equation}
Looking at the corresponding geodesic loop $\overline{\Phi(\gamma_2^{-1})*\Phi(\gamma)*\Phi(\gamma_1^{-1})}$ in the image, we find that it has length:

\[l_2(\Phi(\gamma_1))+ l_2(\Phi(\gamma_2)) +l_2(\Phi(\gamma)) -2l_2(\Phi{\bf p_1})-2l_2(\Phi{\bf p_2})+2d\]
if $\overline{\Phi(\gamma_2)}$ lies to the right of $\overline{\Phi(\gamma_1)}$ or

\[l_2(\Phi(\gamma_1))+ l_2(\Phi(\gamma_2)) +l_2(\Phi(\gamma)) -2d\]
if $\overline{\Phi(\gamma_2)}$ lies to the left of $\overline{\Phi(\gamma_1)}$.  In either case, using that the isomorphism preserves the marked length spectrum, and comparing with equation (\ref{2.4}) we get a contradiction (in the second case, because $d<l_1(\mathbf{p_1})+l_1(\mathbf{p_2})$). 

\vskip 10pt

\noindent {\bf Case (c):}
Finally, the third possibility is that one of $\Phi\mathbf{p_i}$ lies entirely within the other (Figure \ref{fig3}, bottom right).  First, assume that $\overline {\Phi(\gamma_2)}$ is the small inner loop, while $\overline {\Phi(\gamma_1)}$ is the outer loop (so in particular $\Phi{\bf p_2}$ is a subpath of $\Phi{\bf p_1}$).  If we let $0\leq d<l_2(\Phi{\bf p_1})$ be the distance between the left endpoints of $\Phi{\bf p_i}$, a simple calculation will show that:

\[l_2(\overline{\Phi(\gamma_2^{-1})*\Phi(\gamma)*\Phi(\gamma_1^{-1})})= l_2(\overline{\Phi(\gamma_1)})+ l_2(\overline{\Phi(\gamma)})+l_2(\overline{\Phi(\gamma_2)})- 2d-2l_2(\Phi{\bf p_2})\]

\noindent which we can compare with the expression in equation (\ref{2.4}) to again obtain a contradiction.  If $\overline{\Phi(\gamma_1)}$ is the inner loop, let $0\leq d <l_2(\Phi\mathbf{p_2})$ be the distance between the right endpoints of $\Phi(\mathbf{p_i})$.  We calculate:

\[l_2(\overline{\Phi(\gamma_2^{-1})*\Phi(\gamma)*\Phi(\gamma_1^{-1})})= l_2(\overline{\Phi(\gamma_1)})+ l_2(\overline{\Phi(\gamma)})+l_2(\overline{\Phi(\gamma_2)})- 2d-2l_2(\Phi{\bf p_1})\]
providing the same contradiction.

\vskip 10pt

This gives us that the image paths $\Phi {\bf p_1}$ and $\Phi {\bf p_2}$ are subpaths of the geodesic loop $\gamma$ which agree at one endpoint, but not in any larger neighborhood of the endpoint.  Since $\gamma$ is cyclically reduced, this immediately forces the concatenation $\Phi {\bf p_1}*\Phi {\bf p_2}$
to be a geodesic path, hence geodesically incident paths map to geodesically incident paths. It is now clear that if ${\bf p}={\bf q_2}*{\bf q_1}$ is a geodesic path written as a concatenation of subpaths (all in $\mathcal D(X_1)$), then $\Phi{\bf p}=\Phi{\bf q_2}*\Phi{\bf q_1}$ (since the ${\bf q_i}$ are geodesically incident).

\begin{Claim}
Let ${\bf p_1}$ and ${\bf p_2}$ be a pair of incident geodesic paths in $\mathcal D(X_1)$, meeting at a common vertex $q$ which we will take as the basepoint for $\pi_1(X)$. Then the corresponding geodesic paths $\Phi{\bf p_1}$ and $\Phi{\bf p_2}$ are also incident.
\end{Claim}

To see this, note that if the incident paths are {\it geodesically} incident, we are done by the previous claim.  So let us assume not.  Then by Lemma \ref{concatenation}, we have that the reduced path corresponding to the concatenation ${\bf p_2}*{\bf p_1}$ is of the form ${\bf r_2}*{\bf r_1}$, where ${\bf r_i}$ is a subpath of ${\bf p_i}$.  Furthermore, ${\bf p_2}={\bf r_2}*{\bf q}^{-1}$, while ${\bf p_1}={\bf q}*{\bf r_1}$. We now argue that the geodesics ${\bf q}, {\bf r_1}, {\bf r_2}$ all lie in $\mathcal D(X_1)$. Indeed, assume that $(\gamma_1, \gamma_2)$ is a ${\bf p_1}$-distinguishing pair, and $(\eta_1, \eta_2)$ is a ${\bf p_2}$-distinguishing pair. Then we can immediately write out distinguishing pairs as follows:

%\todo{I changed the distinguishing pairs to $(\gamma_1, \gamma_2)$ and $(\eta_1, \eta_2)$, in order to match the notation we've been using elsewhere. Also changed the layout below a little bit. Please double-check the distinguishing pairs listed out below.}

\begin{itemize}
	\item a distinguishing pair for ${\bf q}$ is given by $(\eta_1^{-1},\overline{\gamma_1 * \eta_2^{-1}})$.
	\item a distinguishing pair for ${\bf r_1}$ is given by either (a) $(\gamma_1, \overline{\gamma_2*\eta_1})$ if $\gamma_2*\eta_1$ is not reducible as a concatenation of paths based at $q$ (it is cyclically reducible), or (b) $(\gamma_1, \overline{\gamma_2*\eta_2})$ if $\gamma_2*\eta_2$ is not reducible as a concatenation of paths based at $q$.  (Note that these are mutually exclusive options.)
	
	\item a distinguishing pair for ${\bf r_2}$ is given by either (a) $(\eta_1, \overline{\gamma_1*\eta_2})$ if $\gamma_1*\eta_2$ is not reducible as a concatenation of paths based at $q$, or (b) $(\eta_1, \overline{\gamma_2*\eta_2})$ if $\gamma_2*\eta_2$ is not reducible as a concatenation of paths based at $q$. (Again, these are mutually exclusive options.)
\end{itemize}

As usual, the overline means one should cyclically reduce. Above we have been somewhat cavalier with the starting points of the various loops.  They should all be reparametrized so that they are based at the appropriate points (initial point of ${\bf q}$, ${\bf r_1}$, ${\bf r_2}$ respectively), by shifting the basepoint along the paths ${\bf q}$, ${\bf r_1}$, and/or ${\bf r_2}$ as needed.

Since ${\bf q}$ is geodesically incident to both ${\bf r_1, r_2}$, the previous claim implies that $\Phi{\bf p_2}=\Phi{\bf r_2}*\Phi{\bf q}^{-1}$ and $\Phi{\bf p_1}=\Phi{\bf q}*\Phi{\bf r_1}$.  However, reversal of paths is preserved under the map $\Phi$ we have constructed (since we can take the same pair of distinguishing loops with reversed orientations).  This immediately yields our last claim.

\vskip 5pt

We can now complete the proof of the Proposition. We define a map $f:\mathcal B(X_1) \rightarrow B(X_2)$
 as follows: given a point $x\in \mathcal B(X_1)$, we consider the subset $\mathcal D(x) \subset \mathcal D(X_1)$ consisting of all distinguished paths which originate at $x$. Corollary \ref{disting-from-branch-point} guarantees that $\mathcal D(x)\neq \emptyset$. We can apply the map $\Phi$ to all the elements in $\mathcal D(x)$, obtaining a subset $\Phi\big( \mathcal D(x) \big)$ of $\mathcal D(X_2)$. In view of Claim 6, all the distinguished paths given by the elements $\Phi\big( \mathcal D(x) \big)$ originate at the same point in $\mathcal B(X_2)$. We define this point to be $f(x)$, i.e. $f(x)\in X_2$ is the unique point with the property $\Phi\big( \mathcal D(x) \big) \subset \mathcal D(f(x))$. 
 
Next we argue that the map $f$ is distance non-increasing. If $x\neq y$ are a pair of distinct points in $\mathcal B(X_1)$, we let ${\bf p}$ denote a distance minimizer from $x$ to $y$. Corollary \ref{dist-min-are-disting} tells us that ${\bf p}, {\bf p}^{-1} \in \mathcal D(X_1)$. This gives us elements ${\bf p}\in \mathcal D(x)$ and ${\bf p}^{-1} \in \mathcal D(y)$, so by definition of the map $f$, the image path $\Phi {\bf p}$ is a geodesic path originating at $f(x)$ and ending at $f(y)$. But from Claim 2, we know that the map $\Phi$ preserves the length of paths. We conclude that 
$$d\big(f(x), f(y)\big) \leq l_2\big(\Phi {\bf p}\big) = l_1({\bf p}) = d(x,y)$$

Applying the same argument to the reverse isomorphism $\pi_1(X_2)\rightarrow \pi_1(X_1)$, we have the inverse map $\Psi:\mathcal D(X_2)\rightarrow \mathcal D(X_1)$. Applying the construction described above, we obtain an induced map $g: \mathcal B(X_2)\rightarrow B(X_1)$. The argument in the previous paragraph tells us that $g$ is also distance non-increasing, and by construction, we have that $f, g$ are inverse maps of each other (compare with the discussion in Claim 4). Composing the two maps, we obtain
$$d(x,y) = d\big(g(f(x)), g(f(y))\big) \leq d(f(x), f(y)) \leq d(x,y)$$
Hence all the inequalities are actually equalities, and the maps $f, g$ are isometries. This completes the proof of Proposition \ref{isombranch}
\end{proof}

We now prove the main theorem of this section:

\begin{proof}[Proof of Theorem \ref{special-case}]
Consider the set $\mathcal B(X_1)$. If this set is empty, then Proposition \ref{isombranch} forces the set $\mathcal B(X_2)$ to likewise be empty. Lemma \ref{circle} tells us that each $X_i$ is isometric to a circle $S^1$ of some radius $r_i$. The function $l_i$, applied to one of the generators of $\pi_1(X_i) \cong \mathbb Z$ evaluates to $2\pi r_i$. Since the marked length spectrum is preserved, we conclude $2\pi r_1= 2\pi r_2$, and the two circles are isometric.

So now we may assume that $\mathcal B(X_1)\neq \emptyset$. By Proposition \ref{isombranch}, we already know that there is an isometry $f: \mathcal B(X_1)\rightarrow \mathcal B(X_2)$ between the sets of branch points.  All we need to show is that we can extend this isometry to a global isometry. Note that, since we are working with compact (hence complete) metric spaces, an isometry between the subsets $\mathcal B(X_i)$ extends to an isometry $f: \overline{\mathcal B(X_1)}\rightarrow \overline{\mathcal B(X_2)}$ between their closures.  We are left with extending our map to points that do not lie in the closure of the branch points.

Let $p\in X_1\setminus \overline{\mathcal B(X_1)}$ lie outside of the closure of the branching locus. We know $p$ lies in a connected component $W$ of $X_1\setminus \overline{\mathcal B(X_1)}$. From Lemma \ref{components-comp-branch}, $W$  equipped with its intrinsic geodesic structure is isometric to an open interval of some finite length $r$, with endpoints lying in the set $\overline{\mathcal B(X_1)}$. Let $x, y \in \overline{\mathcal B(X_1)}$ be the two points to which the interval $W$ gets attached. There are three possibilities, according to whether (i) $x, y \in \mathcal B(X_1)$, (ii) exactly one of the two points $x, y$ lies in $\overline{\mathcal B(X_1)} \setminus \mathcal B(X_1)$, or (iii) both points $x, y$ lie in $\overline{\mathcal B(X_1)} \setminus \mathcal B(X_1)$.  We present a proof that works in all three cases, though in cases (i) and (ii) it can be simplified.

First, extend $W$ to a slightly larger, yet still embedded, geodesic $W'$ whose endpoints $x_1$ and $y_1$ lie in $\mathcal{B}(X_1)$. (In case (i) or (ii) the extension to both or one side, respectively, can be trivial, with corresponding simplifications in the argument to follow.)  By Lemma \ref{extending-W}, there exist branch points $x_i$ and $y_i$ on $W'$ with $x_i \to x$, $y_i \to y$.  Let $W_i$ denote the subpath of $W'$ connecting $x_i$ and $y_i$.  As the $W_i$ are embedded geodesic paths, we may use the construction of Proposition \ref{isombranch} to find corresponding paths $\Phi_{W_i}$ in $X_2$ with lengths equal to $l(W_i)$ connecting $f(x_i)$ and $f(y_i)$.  By continuity of $f$ defined on $\overline{\mathcal{B}(X_1)}$, $f(x_i)\to f(x)$ and $f(y_i)\to f(y)$.  In addition, it is clear from the construction of Proposition \ref{isombranch} that $\Phi_{W_{i+1}} \subseteq \Phi_{W_i}$.  Therefore, the paths $\Phi_{W_i}$ converge to a path $W^*$ in $X_2$ of length $r$ connecting $f(x)$ and $f(y)$.  By Lemma \ref{paths-comp-branch}, $W^*$ coincides with one connected component of $X_2\backslash \overline{\mathcal{B}(X_2)}$ and we can extend the map $f$ to $W$ by isometrically sending $W$ to $W^*$. 

To ensure that this extension of $f$ is well-defined, consider any other geodesic extension of $W$, together with any other choice of the points $x_i, y_i$.  Regardless of these choices, we will still have that $f(x_i)\to f(x)$ and $f(y_i)\to f(y)$.  Then Corollary \ref{unique-comp-attached} ensures that we obtain the same $W^*$ (in the case where both endpoints are in $\mathcal B(X_1)$, use Claim 3 in the proof of Proposition \ref{isombranch} instead).  Therefore, the extension of $f$ is well-defined, and it is straightforward to check that the extension is still an isometry onto its image. 

%\todo{Should we write some more details on why the extension is still an isometry, or is it ok to leave this to the reader?}

The map $f$ we have constructed is an isometric embedding from $Y_1$ into $Y_2$.  But note that we can apply the same  construction to $\phi^{-1}$, yielding an isometric embedding from $Y_2$ into $Y_1$.  Furthermore, the composite of the two maps corresponds to the map from $Y_1$ to itself obtained by applying this construction to the identity isomorphism and hence must be the identity map on $Y_1$.  This implies the map is an isometry from $Y_1$ to $Y_2$.  By the naturality of the construction, we see that the map we constructed induces the isomorphism (up to change of basepoint) between the $\pi_1(X_i)$.  This completes the proof of Theorem \ref{special-case}.

\end{proof}

As discussed at the beginning of this section, our {\bf Main Theorem} now follows immediately from the special case established in Theorem \ref{special-case}.

%
%%
%%%
%%%%
%%%%%
%%%%%%%%%%%%%%%%%%%%%%%%%

\bibliographystyle{alpha}
\bibliography{1dmls}

\end{document}